\documentclass[10pt]{article}

\usepackage[margin=1.5in]{geometry}

\RequirePackage[OT1]{fontenc}
\RequirePackage{amsthm,amsmath,amssymb}
\RequirePackage[numbers]{natbib}
\RequirePackage[colorlinks,citecolor=blue,urlcolor=blue]{hyperref}

\usepackage{subfigure}
\usepackage{graphicx}

\numberwithin{equation}{section}
\theoremstyle{plain}
\newtheorem{thm}{Theorem}[section]
\newtheorem{lemma}{Lemma}[section]
\newtheorem{prop}{Proposition}[section]
\newtheorem{assu}{Assumption}[section]
\newtheorem{defi}{Definition}[section]
\theoremstyle{remark}

\allowdisplaybreaks

\begin{document}

\begin{center}
\Large{A Framework for Controlling Sources of Inaccuracy in Gaussian Process Emulation of Deterministic Computer Experiments}
\end{center}

Benjamin Haaland$^a$, Wenjia Wang$^a$, Vaibhav Maheshwari$^b$\\

\noindent$^a$Georgia Institute of Technology, School of Industrial and Systems Engineering. Emails: bhaaland3@gatech.edu, wenjiawang@gatech.edu\\
$^b$Renal Research Institute. Email: vaibhav.maheshwari@rriny.com\\

%
%
%



\begin{abstract}
Computer experiments have become ubiquitous in science and engineering. Commonly,
runs of these simulations demand considerable time and computing, making experimental
design extremely important in gaining high quality information with limited time and resources.
Broad principles of experimental design are proposed and justified which ensure high nominal,
numeric, and parameter estimation accuracy for Gaussian process emulation of deterministic simulations.
The space-filling properties ``small fill distance'' and ``large separation distance'' are only
weakly conflicting and ensure well-controlled nominal, numeric, and parameter estimation error.
Non-stationarity indicates a greater density of experimental inputs in regions of the input space
with more quickly decaying correlation, while non-constant regression functions indicate a balancing
of traditional design features with space-fillingness. This work provides robust, rigorously
justified, and practically useful overarching sufficient principles for scientists and engineers selecting
combinations of simulation inputs with high information content.
\end{abstract}



\section{Background}\label{introduction}

Computer experiments are complex mathematical models implemented in large computer codes used by scientists and engineers to study real systems. 
In many situations, the potential for actual experimentation could be very limited. 
For example, 
a computational fluid dynamics simulation could be used to compare outflow rates of various sclera flap geometries in trabeculectomy \cite{tse}, 
mosquito population dynamics could be coupled with dengue transmission models to study urban dengue control \cite{ellis}, or
a finite-volume radiation hydrodynamics model could be calibrated to a small physical data set, allowing exploration of input configurations for which experimental data is not available \cite{holloway}. Often, a thorough exploration of the unknown simulation function or mean simulation function is wanted. However, the simulation is typically expensive enough that this exploration must be conducted very wisely. 
A seemingly high-quality solution is to evaluate the expensive simulation at several well-distributed data sites and then build an inexpensive approximation, or emulator, for the simulation. The accuracy of this emulator depends very strongly on the manner in which data is collected from the expensive function \cite{santner, fang, haaland2010}. 

Here, we develop principles of \emph{data collection} for Gaussian process emulation of deterministic computer experiments which are broadly applicable and rigorously justified. Three sources of inaccuracy will be considered, nominal error, numeric error, and emulation error due to parameter estimation. Stationary and non-stationary situations, as well as regression functions, will be considered.
Here, the goal is not to develop optimal experimental designs, but instead to identify features of experimental designs which \emph{simultaneously} have high-quality nominal, numeric, and parameter estimation properties across a broad spectrum of specific situations.
Several design criteria are available for computer experiments, including maximin, minimax, low-dimensional stratification, discrepancy, minimum integrated mean squared prediction error, maximum entropy, orthogonality, and others \cite{jin}.
Here, we proceed by decomposing the emulation error into nominal, numeric, and parameter estimation components, then bound each source of error in turn, in terms of design properties. 
For the stationary covariance with constant (or null) mean situation, the developed conditions are well-aligned with the distance-based maximin and minimax criteria.
Non-stationarity indicates more design locations in regions with more rapidly decaying correlation, 
while regression functions introduce a balancing of traditional design properties with the distance-based criteria.
To the best of our knowledge, this is the first presentation of sufficient conditions for simulataneously controlling nominal, numeric, and parameter estimation error in Gaussian process emulation under stationarity, non-stationarity, and the presence of regression functions.

\section{Preliminaries}\label{preliminaries}

Let $f: \Omega\to\mathbb{R}$, denote the function linking a computer experiment's input to its output for $\Omega\subset\mathbb{R}^d$. 
Ordinarily, the approximation $\hat f$ to the unknown function $f$ depends on several parameters $\boldsymbol{\vartheta}$.
Here, let
$\hat{f}_{\boldsymbol{\vartheta}}$ denote the \emph{nominal} emulator at a particular value of the parameters $\boldsymbol{\vartheta}$, and 
$\tilde{f}_{\boldsymbol{\vartheta}}$ denote the \emph{numeric} emulator at a particular value of the parameters $\boldsymbol{\vartheta}$. 
The numeric emulator represents the emulator which is calculated using floating point arithmetic, while the nominal emulator represents the idealized, exact arithmetic, version thereof.
Then, for any norm $\|\cdot\|$, particular value of the parameters $\boldsymbol{\vartheta}_*$, and corresponding parameter estimate $\hat{\boldsymbol{\vartheta}}$, the normed deviation of the emulator from the computer experiment can be decomposed into \emph{nominal}, \emph{numeric}, and \emph{parameter estimation} components using the triangle inequality as shown below, extending ideas in \cite{haaland2011} to also consider inaccuracy due to parameter estimation. 
\begin{gather}
\begin{split}
\|f-\tilde{f}_{\hat{\boldsymbol{\vartheta}}}\|&=\|f-\hat{f}_{\boldsymbol{\vartheta}_*}+
\hat{f}_{\boldsymbol{\vartheta}_*}-\hat{f}_{\hat{\boldsymbol{\vartheta}}}+
\hat{f}_{\hat{\boldsymbol{\vartheta}}}-\tilde{f}_{\hat{\boldsymbol{\vartheta}}}\|\\
&\le\underbrace{\|f-\hat{f}_{\boldsymbol{\vartheta}_*}\|}_{\rm nominal}+
\underbrace{\|\hat{f}_{\hat{\boldsymbol{\vartheta}}}-\tilde{f}_{\hat{\boldsymbol{\vartheta}}}\|}_{\rm numeric}+
\underbrace{\|\hat{f}_{\boldsymbol{\vartheta}_*}-\hat{f}_{\hat{\boldsymbol{\vartheta}}}\|}_{\rm parameter}. \label{decomposition}
\end{split}
\end{gather}
Note that inequality (\ref{decomposition}) does not make any assumption about the norm or type of emulator used. 
It is also noteworthy that this error decomposition considers numeric error in evaluation of the interpolator but does not consider numeric error in the parameter estimation process. 
Explicit consideration of numeric error in parameter estimation would result in a fourth term in the error decomposition.

Consider the $L_2$ norm on the domain of interest $\Omega$ 
\begin{gather}
\|g\|=\|g\|_{L_2(\Omega)}=\sqrt{\int_\Omega g({\bf x})^2{\rm d}{\bf x}}.\label{L2}
\end{gather}
For the $L_2$ norm (\ref{L2}) and any expectation $\mathbb{E}$ we have
\begin{gather}
\begin{split}
\mathbb{E}\|g\|&=\mathbb{E}\sqrt{\int_\Omega g({\bf x})^2{\rm d}{\bf x}}\le\sqrt{\int_\Omega\mathbb{E}g({\bf x})^2{\rm d}{\bf x}},\label{expSqrt}
\end{split}
\end{gather}
by Jensen's inequality and Tonelli's theorem \cite{bartle}.
Applying relation (\ref{expSqrt}) to the error decomposition (\ref{decomposition}), gives
\begin{gather}
\begin{split}
&\mathbb{E}\|f-\tilde{f}_{\hat{\boldsymbol{\vartheta}}}\|
\le \mathbb{E}\|f-\hat{f}_{\boldsymbol{\vartheta}_*}\|+
\mathbb{E}\|\hat{f}_{\hat{\boldsymbol{\vartheta}}}-\tilde{f}_{\hat{\boldsymbol{\vartheta}}}\|+
\mathbb{E}\|\hat{f}_{\boldsymbol{\vartheta}_*}-\hat{f}_{\hat{\boldsymbol{\vartheta}}}\|\\
&\le\sqrt{\int_\Omega \mathbb{E}\left\{f({\bf x})-\hat{f}_{\boldsymbol{\vartheta}_*}({\bf x})\right\}^2{\rm d}{\bf x}} + \sqrt{\int_\Omega \mathbb{E}\left\{\hat{f}_{\hat{\boldsymbol{\vartheta}}}({\bf x})-\tilde{f}_{\hat{\boldsymbol{\vartheta}}}({\bf x})\right\}^2{\rm d}{\bf x}}\\
&\quad+\sqrt{\int_\Omega \mathbb{E}\left\{\hat{f}_{\boldsymbol{\vartheta}_*}({\bf x})-\hat{f}_{\hat{\boldsymbol{\vartheta}}}({\bf x})\right\}^2{\rm d}{\bf x}}.\label{decomposition2}
\end{split}
\end{gather}
More generally, convexity of the function $t\mapsto |t|^p$ for $1\le p\le\infty$ ensures that versions of inequalities (\ref{expSqrt}) and (\ref{decomposition2}) hold for the $L_p(\Omega)$ norms $1\le p\le\infty$.
Here, the \emph{loss function}, or measure of inaccuracy, will be the left-hand-side of equation (\ref{decomposition2}), $\mathbb{E}\|f-\tilde{f}_{\hat{\boldsymbol{\vartheta}}}\|$, for the expectation conditional on the \emph{data}, the experimental design and corresponding output values, $({\bf x}_i,f({\bf x}_i)),$ $i=1,\ldots,n$.

Throughout, consider a Gaussian Process (GP) model for interpolation, 
\begin{gather}
f\sim {\rm GP} \left( {\bf h}(\cdot)'\boldsymbol{\beta},\Psi_{\boldsymbol{\theta}}(\cdot,\cdot) \right)\nonumber
\end{gather}
for some fixed, known regression functions ${\bf h}(\cdot)$, 
and let $\boldsymbol{\vartheta}=\left(\begin{array}{ccc}\boldsymbol{\beta}' & \boldsymbol{\theta}'\end{array}\right)'$. 
The \emph{particular} GP draw $f$ is the \emph{truth} one hopes to discover.
It is assumed that $\Psi_{\boldsymbol{\theta}}(\cdot,\cdot)$ is a positive definite function \cite{wendland2005}.
A considerable collection of theoretical development related to GP regression is given in \cite{steinKriging}.
For a particular dataset ${\bf X}=\{{\bf x}_1,\ldots,{\bf x}_n\}$ and input of interest ${\bf x}$, the \emph{best linear unbiased predictor} (BLUP) is
\begin{equation}
\hat{f}_{\boldsymbol{\vartheta}}({\bf x})={\bf h(x)}'\hat{\boldsymbol{\beta}}+\Psi_{\boldsymbol{\theta}}({\bf x,X})\Psi_{\boldsymbol{\theta}}
({\bf X,X})^{-1} \left( f({\bf X})-{\bf H(X)}\hat{\boldsymbol{\beta}} \right), \label{GPinterpolator}
\end{equation}
where $\Psi_{\boldsymbol{\theta}}({\bf A,B})=\{\Psi_{\boldsymbol{\theta}}({\bf a}_i,{\bf b}_j)\}$ and $f({\bf A})=\{f({\bf a}_i)\}$ for ${\bf A}=\{{\bf a}_i\}$ and ${\bf B}=\{{\bf b}_j\}$, ${\bf H}({\bf X})$ has rows ${\bf h}({\bf x}_i)'$, $\hat{\boldsymbol{\beta}}= \left( {\bf H}({\bf X})'\Psi_{\boldsymbol{\theta}}({\bf X},{\bf X})^{-1}{\bf H}({\bf X}) \right)^{-1}{\bf H}({\bf X})'\Psi_{\boldsymbol{\theta}}({\bf X},{\bf X})^{-1}f({\bf X})$,
and ${\boldsymbol{\theta}}$ equals the vector of true correlation parameters \cite{SWMW1989}.
The BLUP is \emph{best} in the sense that it minimizes the MSPE, $\mathbb{E}\left\{f({\bf x})-\hat{f}_{\boldsymbol{\vartheta}_*}({\bf x})\right\}^2$ (integrated and square rooted in the first term on the right-hand-side of equation (\ref{decomposition2})).
While the BLUP cannot actually be computed, due to unknown correlation parameters and floating point arithmetic, the parameters are commonly estimated via maximum likelihood and the BLUP's floating point approximation is taken as the estimate of the unknown function $f$.
Here, the BLUP, as shown in (\ref{GPinterpolator}), will be taken as the \emph{nominal} emulator.

The overall approach will be to provide bounds for each of the three terms on the right-hand side of (\ref{decomposition2}) in terms of properties of the experimental design ${\bf X}$. 
It will be shown that the \emph{nominal}, \emph{numeric}, and \emph{parameter estimation} criteria are only weakly conflicting, and lead to broadly similar experimental designs.
The remainder of this article is organized as follows. 
In Sections \ref{nominal}, \ref{numeric}, and \ref{paramEst}, bounds on the nominal, numeric, and parameter estimation error, respectively, are developed. 
In each section, experimental design characteristics which lead to small error bounds are discussed and a few examples are given.
In Section \ref{discussion}, a few comparisons are presented and the implications of these broad principles are briefly discussed.

\section{Nominal Error}\label{nominal} 
Focusing on the first term of (\ref{decomposition2}), the \emph{nominal} or mean squared prediction error (MSPE) is given by \cite{SWMW1989},
\begin{equation}
\begin{split}
&\mathbb{E}\left\{f({\bf x})-\hat{f}_{\boldsymbol{\vartheta}}({\bf x})\right\}^2 \\
&=
\Psi_{\boldsymbol{\theta}}({\bf x,x})-\left(\begin{array}{cc}
{\bf h}({\bf x})'&\Psi_{\boldsymbol{\theta}}({\bf x,X})
\end{array}\right)\left(\begin{array}{cc}
{\bf 0}&{\bf H(X)}'\\
{\bf H}({\bf X})&\Psi_{\boldsymbol{\theta}}({\bf X,X})
\end{array}\right)^{-1}\left(\begin{array}{c}
{\bf h}({\bf x})\\
\Psi_{\boldsymbol{\theta}}({\bf X,x})\label{MSPE}
\end{array}\right).
\end{split} 
\end{equation}
Note that throughout this section, the unknown parameters are taken at their true values.
This is not an assumption \emph{per se}, but instead a consequence of the error decomposition (\ref{decomposition}).
The error due to estimating the parameters is considered separately.
In line with intuition, increasing the number of data points always reduces the nominal error. 
The proof of Proposition \ref{nominalProp} is provided in Appendix \ref{nominalPropProof}.
\begin{prop}\label{nominalProp}
If $f\sim {\rm GP} \left( {\bf h}(\cdot)'\boldsymbol{\beta},\Psi_{\boldsymbol{\theta}}(\cdot,\cdot) \right)$, for fixed, known regression functions ${\bf h}(\cdot)$ and ${\bf X}_1\subseteq{\bf X}_2$, then
${\rm MSPE}_2\le{\rm MSPE}_1$,
where ${\rm MSPE}_1$ and ${\rm MSPE}_2$ denote the MSPE of the BLUPs based on ${\bf X}_1$ and ${\bf X}_2$, respectively.
\end{prop}

Notably, this result states that the \emph{expected} squared prediction error is always reduced by the addition of data. 
Certainly, for a particular combination of GP draw $f$ and new design location(s), or for a mis-specified model, the \emph{actual} squared prediction error at a location of interest could in fact be increased by the addition of data.

Consider controlling the inner part of the bound on the nominal error given by the first term of (\ref{decomposition2}). The MSPE in the inner part of the nominal error is given by (\ref{MSPE}).
Applying partitioned matrix inverse results, (\ref{MSPE}) can be rewritten as
\begin{equation}
\begin{split}
&\mathbb{E}\left\{f({\bf x})-\hat{f}_{\boldsymbol{\vartheta}}({\bf x})\right\}^2\\
&=\Psi_{\boldsymbol{\theta}}({\bf x},{\bf x})-\Psi_{\boldsymbol{\theta}}({\bf x},{\bf X})\Psi_{\boldsymbol{\theta}}({\bf X},{\bf X})^{-1}\Psi_{\boldsymbol{\theta}}({\bf X},{\bf x})\\
&\quad + \left({\bf h(x)}-{\bf H}({\bf X})'\Psi_{\boldsymbol{\theta}}({\bf X},{\bf X})^{-1}\Psi_{\boldsymbol{\theta}}({\bf X},{\bf x})\right)'  \left({\bf H}({\bf X})'\Psi_{\boldsymbol{\theta}}({\bf X},{\bf X})^{-1}{\bf H}({\bf X})\right)^{-1} \\
&\qquad \times \left({\bf h}({\bf x})-{\bf H}({\bf X})'\Psi_{\boldsymbol{\theta}}({\bf X},{\bf X})^{-1}\Psi_{\boldsymbol{\theta}}({\bf X},{\bf x})\right). \label{MSPErewritten}
\end{split}
\end{equation}
Initially, the uppermost terms in (\ref{MSPErewritten}), which provide the MSPE for a model with mean zero or no regression functions, are bounded.
We make use of the below theorem which bounds the uppermost terms of the MSPE (\ref{MSPErewritten}) in terms of local bounds. The proof of Theorem \ref{nominalTheorem} is provided in Appendix \ref{nominalTheoremProof}.

\begin{thm}\label{nominalTheorem}
If $A_i,\;i=1,\ldots,n$ is a \emph{covering} of $\Omega$ in the sense that $\Omega\subseteq\cup_{i=1}^n A_i$, $\Psi_{\boldsymbol{\theta}}$ is a positive definite function with $\Psi_{\boldsymbol{\theta}}({\bf x},{\bf x})=\sigma^2$ for all ${\bf x}\in\Omega$, and ${\bf x}_i\in A_i$ for ${\bf x}_i\in{\bf X}$, then
\begin{gather}
\begin{split}
&\sup_{{\bf x}\in\Omega}\Psi_{\boldsymbol{\theta}}({\bf x},{\bf x})-\Psi_{\boldsymbol{\theta}}({\bf x},{\bf X})\Psi_{\boldsymbol{\theta}}({\bf X},{\bf X})^{-1}\Psi_{\boldsymbol{\theta}}({\bf X},{\bf x})\\
&\le\frac{1}{k}\left(\sigma^2-\min_i\inf_{{\bf x}\in A_i}\Psi_{\boldsymbol{\theta}}({\bf x}_i,{\bf x})\right)\left(2k-\sigma^2+\min_i\inf_{{\bf x}\in A_i}\Psi_{\boldsymbol{\theta}}({\bf x}_i,{\bf x})\right),\nonumber
\end{split}
\end{gather}
where $k=n\sup_{{\bf u},{\bf v}\in\Omega}\Psi_{\boldsymbol{\theta}}({\bf u},{\bf v})$.
\end{thm}

While Theorem \ref{nominalTheorem} only assumes $\Psi_{\boldsymbol{\theta}}$ is a positive definite function with $\Psi_{\boldsymbol{\theta}}({\bf x},{\bf x})=\sigma^2$, we examine special two cases of particular interest. 
The first covers many situations where stationarity is assumed, while the second considers a model of non-stationarity in correlation, adapted from \cite{ba}. 
We will label these cases respectively as the {Stationary Model} and {Non-Stationary Model}.
While the considered model of non-stationarity is certainly not all inclusive, it forms a quality approximation in many practical situations.
In the context of this model of non-stationarity, the correlation consists of more rapidly and more slowly decaying components with their weights differing across the input space.
Below, $\varphi(\cdot)$ is taken to be a decreasing function of its non-negative argument.

\subsection{Stationary Model}\label{nominalStationary}

Suppose 
$\Psi_{\boldsymbol{\theta}}({\bf u},{\bf v})=\sigma^2\varphi(\|\boldsymbol{\Theta}({\bf u}-{\bf v})\|_2)$.
Theorem \ref{nominalTheorem} can be used to write the overall bound on the uppermost terms of the MSPE (\ref{MSPErewritten}) in terms of
a \emph{Voronoi} covering \cite{aurenhammer} of $\Omega$ with respect to a \emph{Mahalanobis}-like distance \cite{mahalanobis} ${\rm d}_{\boldsymbol{\Theta}}({\bf u},{\bf v})=\|\boldsymbol{\Theta}({\bf u}-{\bf v})\|_2$ as
\begin{gather}
\frac{\sigma^2}{k}\left(1-\varphi\left(\max_i\sup_{{\bf x}\in V_i(\boldsymbol{\Theta})}{\rm d}_{\boldsymbol{\Theta}}({\bf x}_i,{\bf x})\right)\right)\left(2k-1+\varphi\left(\max_i\sup_{{\bf x}\in V_i(\boldsymbol{\Theta})}{\rm d}_{\boldsymbol{\Theta}}({\bf x}_i,{\bf x})\right)\right),\label{stationaryboundi}
\end{gather}
where $V_i(\boldsymbol{\Theta})=\{{\bf x}\in\Omega\colon{\rm d}_{\boldsymbol{\Theta}}({\bf x},{\bf x}_i)\le {\rm d}_{\boldsymbol{\Theta}}({\bf x},{\bf x}_j)\forall j\ne i\}$ and $k=n\varphi(0)$.

Note that 
\begin{gather}
\max_i \sup_{{\bf x}\in V_i(\boldsymbol{\Theta})}{\rm d}_{\boldsymbol{\Theta}}({\bf x}_i,{\bf x})=\sup_{{\bf x}\in\Omega}\min_i{\rm d}_{\boldsymbol{\Theta}}({\bf x}_i,{\bf x}),\label{filldistance}
\end{gather}
is the \emph{fill distance} with respect to the distance ${\rm d}_{\boldsymbol{\Theta}}$.
So, the supremum of the MSPE over possible inputs, for a GP model with mean zero, can be controlled by demanding that the (potentially) non-spherical fill distance (\ref{filldistance}) is small.
Further, the upper bound (\ref{stationaryboundi}) is minimized if it is achieved uniformly for $i=1,\ldots,n$. 
Importantly, 
a uniform bound on the terms (\ref{stationaryboundi}) is achieved by an experimental design ${\bf X}$ for which all the $\sup_{{\bf x}\in V_i(\boldsymbol{\Theta})}{\rm d}_{\boldsymbol{\Theta}}({\bf x}_i,{\bf x})$ are the same. That is, all the Voronoi cells have the same maximum distance with respect to ${\rm d}_{\boldsymbol{\Theta}}$ from their data point to their edge.

\subsection{Non-Stationary Model}\label{nominalNonStationary}

Suppose $\Psi_{\boldsymbol{\theta}}({\bf u},{\bf v})=\sigma^2\left(\omega_1({\bf u})\omega_1({\bf v})\varphi(\|\boldsymbol{\Theta}_1({\bf u}-{\bf v})\|_2)+\omega_2({\bf u})\omega_2({\bf v})\varphi(\|\boldsymbol{\Theta}_2({\bf u}-{\bf v})\|_2)\right)$.
For the Non-Stationary Model,  assume $\omega_1(\cdot),\omega_2(\cdot)\ge 0$ have Lipschitz continuous derivatives on $\Omega$, $\omega^2_1(\cdot)+\omega^2_2(\cdot)=1$, $\boldsymbol{\Theta}_1,\boldsymbol{\Theta}_2$ are non-singular, and $\lambda_{\rm max}(\boldsymbol{\Theta}'_1 \boldsymbol{\Xi}'_2 \boldsymbol{\Xi}_2 \boldsymbol{\Theta}_1)< 1$, where $\boldsymbol{\Xi}_2=\boldsymbol{\Theta}_2^{-1}$.
The final assumption can be interpreted as $\varphi(\|\boldsymbol{\Theta}_2(\cdot-\cdot)\|_2)$ is narrower than $\varphi(\|\boldsymbol{\Theta}_1(\cdot-\cdot)\|_2)$.
Throughout, we use the notation $\lambda_{\rm max}(\cdot)$ and $\lambda_{\rm min}(\cdot)$ for the maximum and minimum eigenvalues of their (diagonalizable) arguments.
Consider the covering of $\Omega$,
$V^*_i=V_i(\boldsymbol{\Theta}_1)\cup V_i(\boldsymbol{\Theta}_2),\;i=1,\ldots,n$.
Note that $V_i(\boldsymbol{\Theta}_1)$ and $V_i(\boldsymbol{\Theta}_2)$ often do not differ strongly. For example, if $\boldsymbol{\Theta}_2=c\boldsymbol{\Theta}_1$, then $V_i(\boldsymbol{\Theta}_1)=V_i(\boldsymbol{\Theta}_2)$.

Here, take a version of the upper bound in Theorem \ref{nominalTheorem} given by inserting 
a slightly reduced argument  
into the decreasing for $y\ge 0$ function $g(y)=\frac{\sigma^2}{k}(1-y)(2k-1+y)$.
First, note that
\begin{gather}
\begin{split}
&\inf_{{\bf x}\in V^*_i}\left\{\omega_1({\bf x}_i)\omega_1({\bf x})\varphi(\|\boldsymbol{\Theta}_1({\bf x}_i-{\bf x})\|_2)+\omega_2({\bf x}_i)\omega_2({\bf x})\varphi(\|\boldsymbol{\Theta}_2({\bf x}_i-{\bf x})\|_2)\right\}\\
&\ge\inf_{{\bf x}\in V^*_i}\left\{\omega_1({\bf x}_i)\omega_1({\bf x})\varphi\left(\sup_{{\bf x}\in V^*_i}{\rm d}_{\boldsymbol{\Theta}_1}({\bf x}_i,{\bf x})\right)+\omega_2({\bf x}_i)\omega_2({\bf x})\varphi\left(\sup_{{\bf x}\in V^*_i}{\rm d}_{\boldsymbol{\Theta}_2}({\bf x}_i,{\bf x})\right)\right\}.\label{lowerbound1}
\end{split}
\end{gather}
The Lipschitz derivatives of $\omega_1(\cdot),\omega_2(\cdot)$ and Taylor's theorem \cite{nocedal} imply
$\omega_1({\bf x})=\omega_1({\bf x}_i)+R_1({\bf x},{\bf x}_i)$ and
$\omega_2({\bf x})=\omega_2({\bf x}_i)+R_2({\bf x},{\bf x}_i),$
where $|R_1({\bf x},{\bf x}_i)|\le k_1\|{\bf x}_i-{\bf x}\|_2$ and $|R_2({\bf x},{\bf x}_i)|\le k_2\|{\bf x}_i-{\bf x}\|_2$.
The bound (\ref{lowerbound1}) can, in turn, be bounded below as
\begin{gather}
\begin{split}
&\omega^2_1({\bf x}_i)\varphi\left(\sup_{{\bf x}\in V^*_i}{\rm d}_{\boldsymbol{\Theta}_1}({\bf x}_i,{\bf x})\right)+\omega^2_2({\bf x}_i)\varphi\left(\sup_{{\bf x}\in V^*_i}{\rm d}_{\boldsymbol{\Theta}_2}({\bf x}_i,{\bf x})\right)\\
&-\varphi(0)(k_1+k_2)\max_i\sup_{{\bf x}\in V^*_i}\|{\bf x}_i-{\bf x}\|_2,\label{lowerbound2}
\end{split}
\end{gather}
where for tractability the final term in (\ref{lowerbound2}) is bounded 
uniformly across the design space.
Next, consider an experimental design for which the bounds
(\ref{lowerbound2})
are uniform over $i$.
One might expect that regions of the design space with less weight on the global, long range, correlation $\varphi(\|\boldsymbol{\Theta}_1(\cdot-\cdot)\|_2)$ and more weight on the local, short range, correlation $\varphi(\|\boldsymbol{\Theta}_2(\cdot-\cdot)\|_2)$ would require more closely spaced design points, and \emph{vice versa}. 
Roughly speaking this expectation holds true.

Consider two design points ${\bf x}_i$ and ${\bf x}_j$ along with corresponding (\emph{union of}) Voronoi cell sizes 
$\sup_{{\bf x}\in V^*_i}{\rm d}_{\boldsymbol{\Theta}_1}({\bf x}_i,{\bf x})$,
$\sup_{{\bf x}\in V^*_i}{\rm d}_{\boldsymbol{\Theta}_2}({\bf x}_i,{\bf x})$,
$\sup_{{\bf x}\in V^*_j}{\rm d}_{\boldsymbol{\Theta}_1}({\bf x}_j,{\bf x})$, and
$\sup_{{\bf x}\in V^*_j}{\rm d}_{\boldsymbol{\Theta}_2}({\bf x}_j,{\bf x})$.
Suppose that the points in the input space \emph{near} ${\bf x}_i$ have more weight on the global, long range, correlation than the points in the input space \emph{near} ${\bf x}_j$ and the points in the input space \emph{near} ${\bf x}_j$ have more weight on the local, short range, correlation than the points in the input space \emph{near} ${\bf x}_i$, in the sense that
\begin{gather}
\begin{split}
&(-1)^{k-1}\omega_k({\bf x}_i)^2\left(\varphi\left(\sup_{{\bf x}\in V^*_i}{\rm d}_{\boldsymbol{\Theta}_1}({\bf x}_i,{\bf x})\right)-\varphi\left(\sup_{{\bf x}\in V^*_i}{\rm d}_{\boldsymbol{\Theta}_2}({\bf x}_i,{\bf x})\right)\right)\\
&\quad\ge(-1)^{k-1}\omega_k({\bf x}_j)^2\left(\varphi\left(\sup_{{\bf x}\in V^*_j}{\rm d}_{\boldsymbol{\Theta}_1}({\bf x}_j,{\bf x})\right)-\varphi\left(\sup_{{\bf x}\in V^*_j}{\rm d}_{\boldsymbol{\Theta}_2}({\bf x}_j,{\bf x})\right)\right),
\label{nonstationarityAssumption}
\end{split}
\end{gather}
for $k=1,2$. The $(-1)^{k-1}$ terms just mean that the direction of inequality depends on $k$.
Uniformity of the bounds (\ref{lowerbound2}) along with $\omega^2_1(\cdot)+\omega^2_2(\cdot)=1$ implies
\begin{gather}
\begin{split}
&\omega_1({\bf x}_i)^2\left(\varphi\left(\sup_{{\bf x}\in V^*_i}{\rm d}_{\boldsymbol{\Theta}_1}({\bf x}_i,{\bf x})\right)-\varphi\left(\sup_{{\bf x}\in V^*_i}{\rm d}_{\boldsymbol{\Theta}_2}({\bf x}_i,{\bf x})\right)\right)\\
&\quad-
\omega_1({\bf x}_j)^2\left(\varphi\left(\sup_{{\bf x}\in V^*_j}{\rm d}_{\boldsymbol{\Theta}_1}({\bf x}_j,{\bf x})\right)-\varphi\left(\sup_{{\bf x}\in V^*_j}{\rm d}_{\boldsymbol{\Theta}_2}({\bf x}_j,{\bf x})\right)\right)\\
&\quad\quad=\varphi\left(\sup_{{\bf x}\in V^*_j}{\rm d}_{\boldsymbol{\Theta}_2}({\bf x}_j,{\bf x})\right)-\varphi\left(\sup_{{\bf x}\in V^*_i}{\rm d}_{\boldsymbol{\Theta}_2}({\bf x}_i,{\bf x})\right)\quad{\rm and}\\
&\omega_2({\bf x}_j)^2\left(\varphi\left(\sup_{{\bf x}\in V^*_j}{\rm d}_{\boldsymbol{\Theta}_1}({\bf x}_j,{\bf x})\right)-\varphi\left(\sup_{{\bf x}\in V^*_j}{\rm d}_{\boldsymbol{\Theta}_2}({\bf x}_j,{\bf x})\right)\right)\\
&\quad-
\omega_2({\bf x}_i)^2\left(\varphi\left(\sup_{{\bf x}\in V^*_i}{\rm d}_{\boldsymbol{\Theta}_1}({\bf x}_i,{\bf x})\right)-\varphi\left(\sup_{{\bf x}\in V^*_i}{\rm d}_{\boldsymbol{\Theta}_2}({\bf x}_i,{\bf x})\right)\right)\\
&\quad\quad=\varphi\left(\sup_{{\bf x}\in V^*_j}{\rm d}_{\boldsymbol{\Theta}_1}({\bf x}_j,{\bf x})\right)-\varphi\left(\sup_{{\bf x}\in V^*_i}{\rm d}_{\boldsymbol{\Theta}_1}({\bf x}_i,{\bf x})\right).\label{uniformity}
\end{split} 
\end{gather}
Combining (\ref{nonstationarityAssumption}) with (\ref{uniformity}) gives
\begin{gather}
\begin{split}
&\sup_{{\bf x}\in V^*_j}{\rm d}_{\boldsymbol{\Theta}_1}({\bf x}_j,{\bf x})\le
\sup_{{\bf x}\in V^*_i}{\rm d}_{\boldsymbol{\Theta}_1}({\bf x}_i,{\bf x})\quad{\rm and}\quad
\sup_{{\bf x}\in V^*_j}{\rm d}_{\boldsymbol{\Theta}_2}({\bf x}_j,{\bf x})\le
\sup_{{\bf x}\in V^*_i}{\rm d}_{\boldsymbol{\Theta}_2}({\bf x}_i,{\bf x}),\nonumber
\end{split}
\end{gather}
since $\varphi$ is a decreasing function of its non-negative argument.
That is, a uniform bound on (\ref{lowerbound2}) is achieved by an experimental design ${\bf X}$ which has smaller (\emph{union of}) Voronoi cells, with respect to either ${\rm d}_{\boldsymbol{\Theta}_1}$ or ${\rm d}_{\boldsymbol{\Theta}_2}$, in regions with more emphasis on the local, more quickly decaying, correlation and less emphasis on the global, more slowly decaying, correlation. Note that the global ($k=1$) and local ($k=2$) \emph{emphases} at ${\bf x}_i$ are given concretely by
\begin{gather}
\begin{split} 
&\omega_k({\bf x}_i)^2\left(\varphi\left(\sup_{{\bf x}\in V^*_i}{\rm d}_{\boldsymbol{\Theta}_1}({\bf x}_i,{\bf x})\right)-\varphi\left(\sup_{{\bf x}\in V^*_i}{\rm d}_{\boldsymbol{\Theta}_2}({\bf x}_i,{\bf x})\right)\right)\quad k=1,2.
\nonumber 
\end{split}
\end{gather}

The assumption that $\omega_1,\omega_2$ have Lipschitz continuous derivatives on $\Omega$, while not overly restrictive in most practical situations, is not necessary, in principle.
Without this assumption, the bilinear form on the right-hand side of (\ref{lowerbound1}) can be bounded below via one of several reverses of the Cauchy-Schwarz inequality \cite{dragomir}.
Many of these results provide a lower bound for (\ref{lowerbound1}) in terms of a geometric mean across ${\bf x}$ and ${\bf x}_i$ of terms similar to (\ref{lowerbound2}), without the terms involving the Lipschitz constants, 
and in turn more complex development for the supremum of the uppermost terms of the MSPE (\ref{MSPErewritten}).
Balancing simplicity and broad applicability, these type of results are not pursued here.

\subsection{Regression Functions}\label{nominalRegFunc}

Subsections \ref{nominalStationary} (Stationary Model) and \ref{nominalNonStationary} (Non-Stationary Model)
relate to the uppermost terms in equation (\ref{MSPErewritten}), which without further development provides the MSPE for a model with mean zero. Now, we consider the lowermost terms in (\ref{MSPErewritten}), which are relevant for Gaussian process models with a mean or non-null regression component.
The lowermost terms in (\ref{MSPErewritten}) can be bounded above as
\begin{gather}
\begin{split}
&\left( {\bf h(x)}-{\bf H(X)}'\Psi_{\boldsymbol{\theta}}({\bf X,X})^{-1}\Psi_{\boldsymbol{\theta}}({\bf X,x}) \right)'
\left( {\bf H(X)} '\Psi_{\boldsymbol{\theta}}({\bf X,X})^{-1}{\bf H(X)} \right)^{-1}\\
&\quad \times \left( {\bf h(x)}-{\bf H(X)}'\Psi_{\boldsymbol{\theta}}({\bf X,X})^{-1}\Psi_{\boldsymbol{\theta}}({\bf X,x}) \right)\\
&\le\left.\left\|{\bf h(x)}-{\bf H(X)}'\Psi_{\boldsymbol{\theta}}({\bf X,X})^{-1}\Psi_{\boldsymbol{\theta}}({\bf X,x})\right\|_2^2\right/
\lambda_{\rm min}\left( {\bf H(X)}'\Psi_{\boldsymbol{\theta}}({\bf X,X})^{-1}{\bf H(X)} \right)\\
&\le\left.n\sup_{{\bf u},{\bf v}\in\Omega}\Psi_{\boldsymbol{\theta}}({\bf u},{\bf v})\left\|{\bf h}({\bf x})-{\bf H}({\bf X})'\Psi_{\boldsymbol{\theta}}({\bf X},{\bf X})^{-1}\Psi_{\boldsymbol{\theta}}({\bf X},{\bf x})\right\|_2^2\right/\lambda_{\rm min}\left({\bf H}({\bf X})'{\bf H}({\bf X})\right).\label{lowermostbound}
\end{split}
\end{gather}
The first inequality is true because ${\bf a}'{\bf B}^{-1}{\bf a}\le\lambda_{\rm max}({\bf B}^{-1})\|{\bf a}\|_2^2$ and $\lambda_{\rm max}({\bf B}^{-1})=1/\lambda_{\rm min}({\bf B})$ and
the second inequality is true because $\lambda_{\rm min}({\bf A}'{\bf B}^{-1}{\bf A})\ge
\lambda_{\rm min}({\bf A}'{\bf A})/\lambda_{\rm max}({\bf B})$ and as implication (\ref{gershgorin}) of Gershgorin's theorem \cite{varga}. Note that the term $n\sup_{{\bf u},{\bf v}\in\Omega}\Psi_{\boldsymbol{\theta}}({\bf u},{\bf v})$ does not depend on the experimental design.

The components of the squared Euclidean norm $\left\|{\bf h}({\bf x})-{\bf H}({\bf X})'\Psi_{\boldsymbol{\theta}}({\bf X},{\bf X})^{-1}\Psi_{\boldsymbol{\theta}}({\bf X},{\bf x})\right\|_2^2$ are squared errors for an interpolator of the regression functions. Intuitively, we might expect these squared interpolation errors to behave in a manner similar to the MSPE for the Gaussian process model with mean zero. In fact, it has been shown above that if the regression functions are draws from a Gaussian process with mean zero and a covariance structure as described in 
Subsections \ref{nominalStationary} and \ref{nominalNonStationary},
then the expectation of these squared errors can be controlled through the experimental design as described above. That is, an experimental design which gives low nominal error in the mean zero case will also make the term $\left\|{\bf h}({\bf x})-{\bf H}({\bf X})'\Psi_{\boldsymbol{\theta}}({\bf X},{\bf X})^{-1}\Psi_{\boldsymbol{\theta}}({\bf X},{\bf x})\right\|_2^2$ small.

Alternatively, a reproducing kernel Hilbert space (RKHS) \cite{aronszajn, wendland2005} may be defined as the completion of the function space 
spanned by $\{\Psi_{\boldsymbol{\theta}}({\bf x}_i,\cdot)\colon {\bf x}_i\in\Omega\}$ with respect to the inner product $\langle\sum_i a_i\Psi_{\boldsymbol{\theta}}({\bf x}_i,\cdot),\sum_j b_j\Psi_{\boldsymbol{\theta}}({\bf y}_j,\cdot)\rangle=\sum_{i,j}a_i b_j\Psi_{\boldsymbol{\theta}}({\bf x}_i,{\bf y}_j)$.
Many commonly selected regression functions, for example constant, linear, polynomial, and spline, will also lie in the RKHSs induced by many of the common covariance functions. For example, the Gaussian kernel induces an RKHS of functions with infinitely many continuous derivatives and Mat\'ern kernels induce RKHSs of functions with a fixed number of continuous derivatives.
If the selected regression functions lie in the RKHS induced by the chosen covariance function, then deterministic RKHS interpolation error bounds as a decreasing function of the fill distance, such as Theorem 5.1 in \cite{haaland2011}, can be applied. Another alternative would be to choose as regression functions covariance function (half) evaluations $\{\Psi_{\boldsymbol{\theta}}({\bf x}_i,\cdot)\colon i\in\mathcal{I}\}$ at a well-distributed set of centers $\mathcal{I}$. These regression functions are capable of \emph{approximating} a broad range of mean functions and have the appealing feature that the lowermost term in the bound (\ref{MSPErewritten}) is then identically zero.

The eigenvalue $\lambda_{\rm min}\left({\bf H}({\bf X})'{\bf H}({\bf X})\right)$ has approximation
\begin{gather}
\lambda_{\rm min}\left({\bf H}({\bf X})'{\bf H}({\bf X})\right)=\lambda_{\rm min}\left(\sum_{i=1}^n {\bf h}({\bf x}_i){\bf h}({\bf x}_i)'\right)\approx n\lambda_{\rm min}\left(\int {\bf h}({\bf y}){\bf h}({\bf y})'{\rm d}F({\bf y})\right)=n s_1,\nonumber
\end{gather}
where $F$ denotes the large sample distribution of the input locations ${\bf X}$, $s_1\ge 0$, and $s_1>0$ unless ${\bf h}({\bf y})'{\bf a}=0$ with probability $1$ with respect to the large sample distribution $F$ for some ${\bf a}\ne 0$.
The (approximate) term 
$s_1=\lambda_{\rm min}\left(\int {\bf h}({\bf y}){\bf h}({\bf y})'{\rm d}F({\bf y})\right)$
in the denominator of (\ref{lowermostbound}) indicates that (at least for regression functions which do not make ${\bf h}({\bf x})-{\bf H}({\bf X})'\Psi_{\boldsymbol{\theta}}({\bf X},{\bf X})^{-1}\Psi_{\boldsymbol{\theta}}({\bf X},{\bf x})\equiv {\bf 0}$), the design properties implied by the mean zero development in 
Subsections \ref{nominalStationary} and \ref{nominalNonStationary}
need to be balanced with traditional experimental design properties. Two common scenarios are of particular interest. First, consider a constant regression function, a mean parameter. In this situation, $s_1=1$ irrespective of experimental design. 
Second, consider linear regression functions in each dimension in addition to the constant. 
If each linear function is expressed on the same scale, 
$s_1$
will be large for an experimental design 
with points far from 
$\boldsymbol{\mu}=\int {\bf y}{\rm d}F({\bf y})$
the average design value, and whose orientations ${\bf x}_i-\boldsymbol{\mu}$ emphasize each basis vector in an orthonormal basis of $\mathbb{R}^d$ equally. 
For the common situation where $\Omega=[0,1]^d$, 
$s_1$
will be maximized for a design with equal numbers of points in each of the corners of $[0,1]^d$.
So, high-quality experimental designs for Gaussian process models with linear regression mean components will balance the fill distance-based criteria described 
in Subsections \ref{nominalStationary} and \ref{nominalNonStationary}
with the push of design points to the ``corners'' of $\Omega$.
Similarly, high-quality experimental designs for Gaussian process models with quadratic regression mean components will balance the fill distance-based criteria described 
in Subsections \ref{nominalStationary} and \ref{nominalNonStationary}
with the push of design points to the edges and \emph{middle} of the design space.


Example 
high quality
23 run experimental designs for the \emph{nominal} situations described in 
Subsections \ref{nominalStationary} (Stationary Model), \ref{nominalNonStationary} (Non-Stationary Model), and
the Stationary Model 
along with linear regression functions are illustrated in the left, middle, and right panels, respectively, of Figure \ref{Design}. 
For each case, $\varphi(d)={\rm exp}\{-d^2\}$.
For the Non-Stationary Model example,
$\omega_1({\bf u})^2=1-\|{\bf u}\|^2/2$, $\omega_2({\bf u})^2=\|{\bf u}\|^2/2$,
$\boldsymbol{\Theta}_1=1\cdot{\bf I}_2$, $\boldsymbol{\Theta}_2=10\cdot{\bf I}_2$, and $\sigma^2=1$, while for 
the Stationary Model
along with linear regression functions example $\boldsymbol{\Theta}=2\cdot{\bf I}_2$ and $\sigma^2=1$.
As expected, in the first panel, illustrating the stationary situation, the design points lie near a triangular lattice (subject to edge effects).
Similarly, in the second panel, illustrating the non-stationary correlation situation, the design points in the upper right, where the shorter range, more quickly decaying, correlation is emphasized, are more dense than in the lower left, where the longer range, more slowly decaying, correlation is emphasized.
Further, in the third panel, illustrating the impact of regression functions, the design points balance fill distance and a push towards the corners of the input space.

Finding designs which minimize (or nearly minimize) the error bounds is challenging. 
Here, we adopted a homotopy continuation \cite{homotopy} approach, which slowly transitions from an easier objective function to a more difficult objective function. 
Note that the respective target objectives for the stationary, non-stationary, and stationary with regression functions situations are given by minimizing equation (\ref{stationaryboundi}), maximizing equation (\ref{lowerbound2}), and minimizing the sum of equations (\ref{stationaryboundi}) and (\ref{lowermostbound}).
For the stationary situation, the optimization routine was initialized at a triangular lattice, crudely scaled for small fill distance.
The objective function is then taken as $(1-\delta)\;{\rm mean}_i \sup_{{\bf x}\in V_i(\boldsymbol{\Theta})}{\rm d}_{\boldsymbol{\Theta}}({\bf x}_i,{\bf x})+\delta\max_i \sup_{{\bf x}\in V_i(\boldsymbol{\Theta})}{\rm d}_{\boldsymbol{\Theta}}({\bf x}_i,{\bf x})$ for $\delta\in[0,1]$. 
Initially, the optimization is performed for $\delta=0$, then $\delta=1/K$, then $\delta=2/K$,and so on, up to $\delta=K/K=1$, for a moderately large number $K$. 
Notice that the \emph{target} objective function is optimized when $\delta=1$.

For both the non-stationary and stationary with regression functions situations, 
the objective function transition occurs both from mean to minimum/maximum and from stationary to non-stationary or with regression functions.
For convenience and without loss of generality, suppose one is minimizing the maximum of local bounds.
Let $Q_\delta({\bf x}_i)$ denote an objective function \emph{local} to data point ${\bf x}_i$, which is continuous as a function of $\delta$ and converges to the \emph{target} local objective function as $\delta\to 1$. 
The objective function is then taken as $(1-\delta)\,{\rm mean}_i\, Q_\delta({\bf x}_i)+\delta\,{\rm max}_i\, Q_\delta({\bf x}_i)$, and is solved repeatedly for a sequence of $\delta$ going from $0$ to $1$.
For the non-stationary case (with minus signs to express as minization), 
\begin{gather}
\begin{split}
Q_\delta({\bf x}_i)=&-\omega^2_\delta({\bf x}_i)\varphi\left(\sup_{{\bf x}\in V^*_i}{\rm d}_{\boldsymbol{\Theta}_1}({\bf x}_i,{\bf x})\right)-(1-\omega^2_\delta({\bf x}_i))\varphi\left(\sup_{{\bf x}\in V^*_i}{\rm d}_{\boldsymbol{\Theta}_2}({\bf x}_i,{\bf x})\right)\\
&+\varphi(0)(k_1+k_2)\max_i\sup_{{\bf x}\in V^*_i}\|{\bf x}_i-{\bf x}\|_2,\nonumber
\end{split}
\end{gather}
with $\omega_\delta({\bf x})=(1-\delta)+\delta\sqrt{1-\|{\bf x}\|^2_2/2}$ and Lipschitz constants $k_1=k_2=1/2$,  while for the non-null regression functions case, 
\begin{gather}
\begin{split}
Q_\delta({\bf x}_i)=&\frac{1}{n}\left(1-\varphi\left(\sup_{{\bf x}\in V_i(\boldsymbol{\Theta})}{\rm d}_{\boldsymbol{\Theta}}({\bf x}_i,{\bf x})\right)\right)\left(2n-1+\varphi\left(\sup_{{\bf x}\in V_i(\boldsymbol{\Theta})}{\rm d}_{\boldsymbol{\Theta}}({\bf x}_i,{\bf x})\right)\right)\\
&+\delta n \left.\sup_{{\bf x}\in V_i}\left\|{\bf h}({\bf x})-{\bf H}({\bf X})'\Psi_{\boldsymbol{\theta}}({\bf X},{\bf X})^{-1}\Psi_{\boldsymbol{\theta}}({\bf X},{\bf x})\right\|_2^2
\right/\lambda_{\rm min}\left({\bf H}({\bf X})'{\bf H}({\bf X})\right).\nonumber
\end{split}
\end{gather}
%
%
Nelder-Mead black box optimization along with penalties to enforce input space constraints was used throughout \cite{nocedal}.

\begin{figure}
\centering
\subfigure{\includegraphics[width=1.65in]{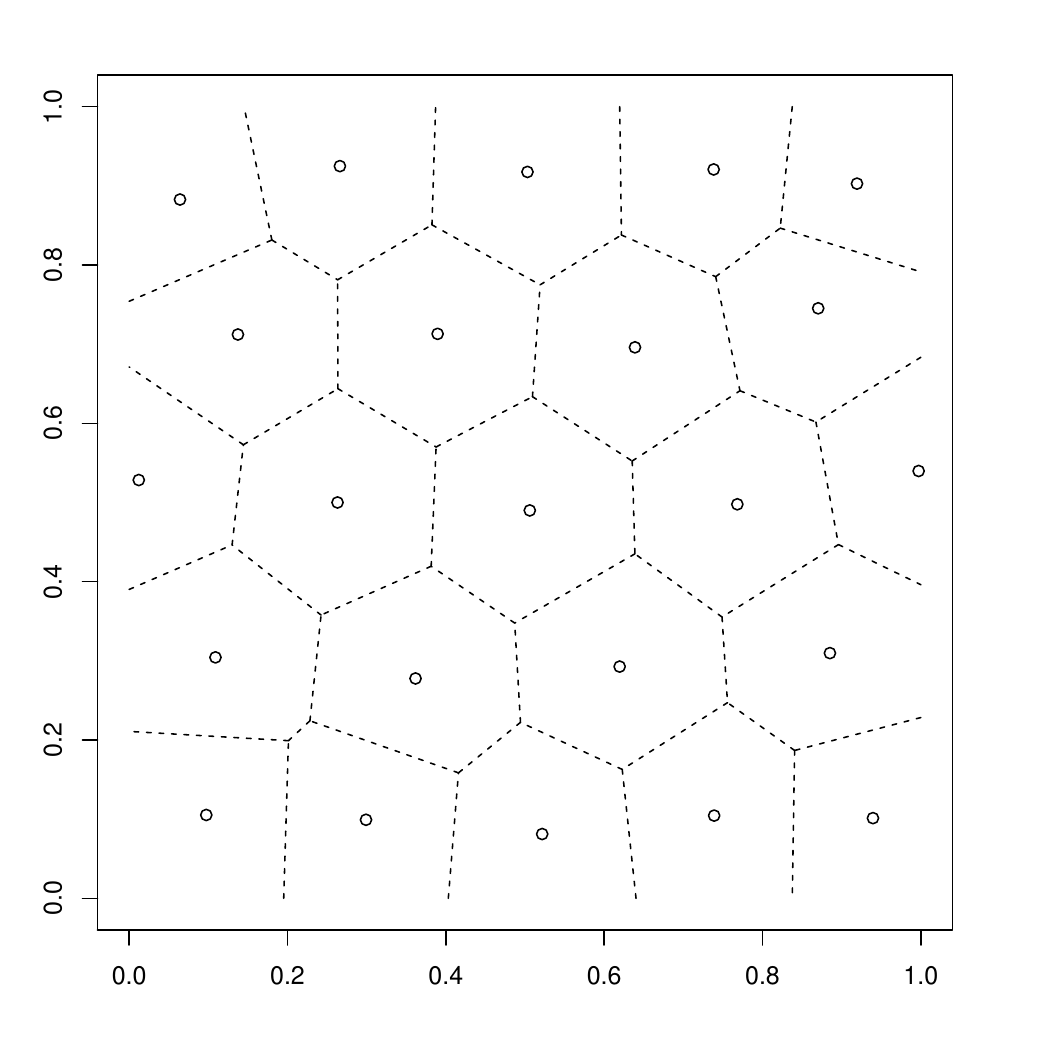}}
\subfigure{\includegraphics[width=1.65in]{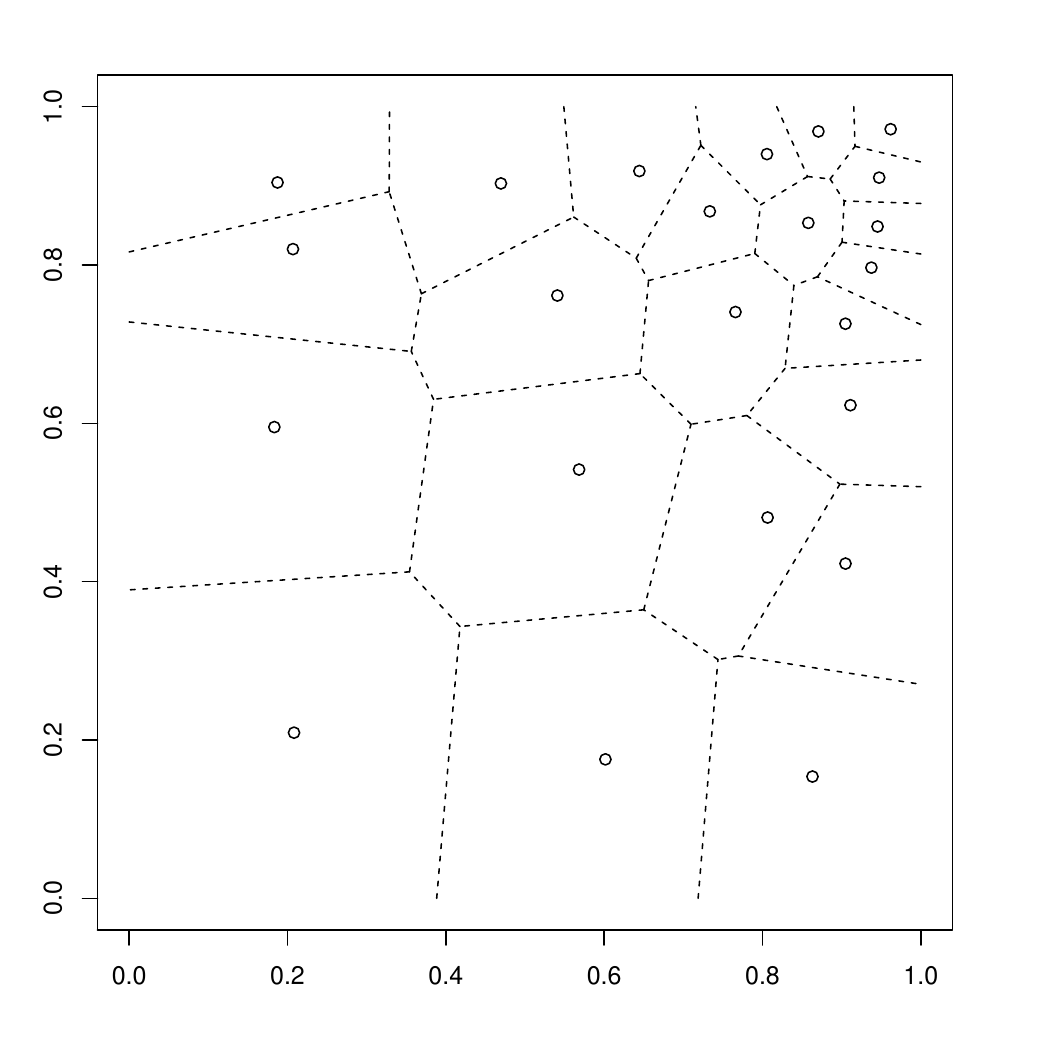}}
\subfigure{\includegraphics[width=1.65in]{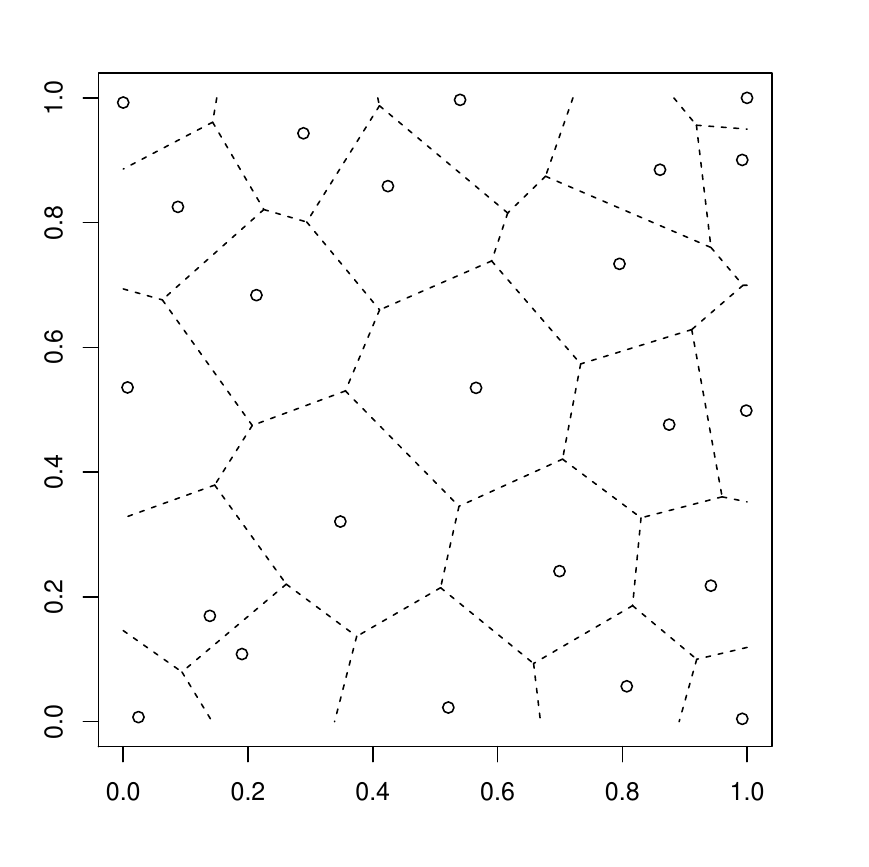}}
\caption{{\bf Left Panel:} Nominal error design for stationary correlation. {\bf Middle Panel:} Nominal error design for the Case 2 model of non-stationary correlation with $\varphi(d)={\rm exp}\{-d^2\}$, 
$\omega_1({\bf u})^2=1-\|{\bf u}\|^2/2$, $\omega_2({\bf u})^2=\|{\bf u}\|^2/2$,
$\boldsymbol{\Theta}_1=1\cdot{\bf I}_2$, $\boldsymbol{\Theta}_2=10\cdot{\bf I}_2$. {\bf Right Panel:} Nominal error design for stationary correlation and a linear regression function for each dimension.}
\label{Design}
\end{figure}

It should be noted that the approach to bounding the MSPE which was used above is not the only approach available. 
One possibility is to formulate the function approximation problem as $n$ numeric integration problems \cite{wozniakowski}, then apply numeric integration results \cite{niederreiter} to provide upper bounds on the \emph{integrated} MSPE in terms of the \emph{star discrepancy} of the point set ${\bf X}$, 
\begin{gather}
D^*({\bf X})=\sup_{J=\prod_{k=1}^d[0,u_k)}\left|\frac{\#\{{\bf x}_i\in{\bf X}:J\cap\Omega\}}{n}-\frac{{\rm vol}\;J\cap\Omega}{{\rm vol}\;\Omega}\right|,\nonumber
\end{gather}
the supremum of differences between the proportion of points in a \emph{rooted} rectangle and the proportion that are supposed to be in the rectangle under uniform measure.
The fill distance can in fact be bounded above in terms of star discrepancy.
It is conceivable that these type of results could be extended to the situation of non-stationarity by considering discrepancy with respect to a more general measure. 
However, this line of development is not pursued here.
We provide the below theorem as an indication of the type of results which are attainable. 
A proof of Theorem \ref{discrepancyThm} and related development is provided in Appendix \ref{discrepancyThmProof}.

\begin{thm}\label{discrepancyThm}
If $f\sim{\rm GP}(0,\Psi(\cdot,\cdot))$, $\Omega=[0,1]^d$, then we have bounds on the integrated mean squared prediction error 
\begin{gather}
\begin{split}
&\mathbb{E}\|f-\hat{f}\|_{L_2([0,1]^d)}^2\le\sum_{j=n+1}^\infty\lambda_j+D^*({\bf X})^2\sum_{j=1}^n\mathbb{E}V(\varphi_i f)^2,\nonumber
\end{split}
\end{gather}
under the condition that $\varphi_i f$ has finite expected squared variation in the sense of Hardy and Kraus $V(\varphi_i f)$ \cite{niederreiter}, where 
$\Psi(\cdot,\cdot)$ has eigenvalue, eigenfunction decomposition $\Psi(\cdot,\cdot)=\sum_{i=1}^\infty\lambda_i\varphi_i(\cdot)\varphi_i(\cdot)$.
\end{thm}

A similar bound can be developed in terms of the \emph{modulus of continuity} \cite{niederreiter} of $\varphi_i f$. 
See Appendix \ref{discrepancyThmProof} for details. 

\section{Numeric Error}\label{numeric}
In Section \ref{nominal}, it has been shown that increasing the number of data points will decrease the nominal error. On the other hand, the numeric error can become arbitrarily large by the addition of new data sites. Here, we develop bounds on the numeric error in terms of properties of the experimental design by adapting and extending results in \cite{golub,wendland2005,haaland2011}. 

The numeric accuracy of Gaussian process emulation depends on the accuracy of floating point matrix manipulations.
Floating point numbers are the rounded versions that computers perform calculations with as opposed to the targeted numbers.
Commonly, computer and software have 15 digits of accuracy meaning that
$\|\tilde{x}-x\|_2/\|x\|_2\le 10^{-15},$
where $x$ denotes the actual value and $\tilde{x}$ denotes the value that the computer stores.
We state a few typical assumptions on computer and software floating point accuracy and provide the following proposition relating numeric accuracy to eigenvalues of $\Psi_{\boldsymbol{\theta}}({\bf X},{\bf X})$. Below, $\kappa(\cdot)=\lambda_{\rm max}(\cdot)/\lambda_{\rm min}(\cdot)$ denotes the condition number of its (diagonalizable) argument. The proof of Proposition \ref{numAssumption} is provided in Appendix \ref{prop31proof}.
\begin{assu}\label{numAssumption}
Take $\kappa(\Psi_{\boldsymbol{\theta}}({\bf X},{\bf X}))=r/\delta$ for $r<1$ and
\begin{gather}
\begin{split}
&\|{\bf h}({\bf x})-\tilde{\bf h}({\bf x})\|_2\le\delta\|{\bf h}({\bf x})\|_2,\;\|f({\bf X})-\tilde{f}({\bf X})\|_2\le\delta\|f({\bf X})\|_2,\;\|h_j({\bf X})-\tilde{h}_j({\bf X})\|_2\le\delta\|h_j({\bf X})\|_2,\\
&\|\Psi_{\boldsymbol{\theta}}({\bf X},{\bf X})-\tilde{\Psi}_{\boldsymbol{\theta}}({\bf X},{\bf X})\|_2\le\delta\|\Psi_{\boldsymbol{\theta}}({\bf X},{\bf X})\|_2,\;{\rm and}\;\|\Psi_{\boldsymbol{\theta}}({\bf x},{\bf X})-\tilde{\Psi}_{\boldsymbol{\theta}}({\bf x},{\bf X})\|_2\le\delta\|\Psi_{\boldsymbol{\theta}}({\bf x},{\bf X})\|_2.\nonumber
\end{split}
\end{gather}
\end{assu}
\begin{prop}\label{prop31}
Under Assumption \ref{numAssumption}, the Gaussian process BLUP (\ref{GPinterpolator}) has numeric error $\left|\hat{f}_{{\boldsymbol{\vartheta}}}({\bf x})-\tilde{f}_{{\boldsymbol{\vartheta}}}({\bf x})\right|$, for arbitrary parameter vector $\boldsymbol{\vartheta}=\left(\begin{array}{cc}\boldsymbol{\beta}'&\boldsymbol{\theta}'\end{array}\right)'$, bounded above by
\begin{gather}
\begin{split}
\delta\|{\bf h}({\bf x})\|_2\|{\boldsymbol{\beta}}\|_2+\frac{2\delta}{1-r}\|\Psi_{\boldsymbol{\theta}}({\bf x},{\bf X})\|_2\left(\|f({\bf X})\|_2+\|{\boldsymbol{\beta}}\|_2\sqrt{\sum_{j=1}^p\|h_j({\bf X})\|_2^2}\right)g({\bf X},\Psi_{\boldsymbol{\theta}}),\label{numInequal4}
\end{split}
\end{gather}
where
\begin{gather}
g({\bf X},\Psi_{\boldsymbol{\theta}})=\frac{\kappa(\Psi_{\boldsymbol{\theta}}({\bf X},{\bf X}))+1}  {\lambda_{\rm min}(\Psi_{\boldsymbol{\theta}}({\bf X},{\bf X}))}.\nonumber
\end{gather}
\end{prop}
For experimental designs which are not too small and have reasonable parameter estimation properties, $\|\hat{\boldsymbol{\beta}}\|_2$ will be of a similar magnitude to $\|\boldsymbol{\beta}\|_2$.
Further, the terms $\|\Psi_{\boldsymbol{\theta}}({\bf x},{\bf X})\|_2$, $\|f({\bf X})\|_2$, and $\|h_j({\bf X})\|_2$ are Monte Carlo approximations to  $\sqrt{n}\|\Psi_{\boldsymbol{\theta}}({\bf x},\cdot)\|_{L_2(F)}$, $\sqrt{n}\|f(\cdot)\|_{L_2(F)}$, and $\sqrt{n}\|h_j(\cdot)\|_{L_2(F)}$, respectively, with respect to the large sample distribution of the experimental design $F$.
That is, the terms in the bound (\ref{numInequal4}), aside from $g({\bf X},\Psi_{\boldsymbol{\theta}})$, influence the numeric accuracy only weakly and vanishingly, and the bound depends on the experimental design primarily through $g({\bf X},\Psi_{\boldsymbol{\theta}})$.
The implication of Gershgorin's theorem \cite{varga} in (\ref{gershgorin}) implies $g({\bf X},\Psi_{\boldsymbol{\theta}})$ can be bounded in terms of the minimum eigenvalue of $\Psi_{\boldsymbol{\theta}}({\bf X},{\bf X})$ as
\begin{gather}\label{g}
g({\bf X},\Psi_{\boldsymbol{\theta}})\le\frac{1}{\lambda_{\rm min}(\Psi_{\boldsymbol{\theta}}({\bf X},{\bf X}))}\left(\frac{n\sup_{{\bf u},{\bf v}\in\Omega}\Psi_{\boldsymbol{\theta}}({\bf u},{\bf v})}{\lambda_{\rm min}(\Psi_{\boldsymbol{\theta}}({\bf X},{\bf X}))}+1\right).
\end{gather}

We adapt and generalize results from \cite{wendland2005} in the theorem below providing a lower bound for $\lambda_{\rm min}(\Psi_{\boldsymbol{\theta}}({\bf X},{\bf X}))$ and thereby an upper bound for (\ref{g}). 
The proof of Theorem \ref{numericTheorem} is provided in Appendix \ref{numericTheoremProof}.
First, a definition of the Fourier transform is provided.
\begin{defi}
For $f\in L_1(\mathbb{R}^d)$ define the Fourier transform \cite{stein1971}
\begin{gather}
\hat{f}(\omega)=(2\pi)^{-d/2}\int_{\mathbb{R}^d}f(x)e^{-i\omega'x}{\rm d}x.\nonumber
\end{gather}
\end{defi}
\begin{thm}\label{numericTheorem}
Suppose $\Phi$ is a positive definite, translation invariant kernel with Fourier transform $\hat{\Phi}\in L_1(\mathbb{R}^d)$. Then,
\begin{gather}
\sum_{j=1}^n\sum_{k=1}^n\alpha_j\alpha_k\Phi({\bf x}_j-{\bf x}_k)\ge\Upsilon_{c_*/q(\boldsymbol{\Theta})}({\bf 0})\sum_{j=1}^n\alpha_j^2\left(1-\frac{\Gamma^2(d/2+1)\pi}{18}\left(\frac{q(\boldsymbol{\Theta})}{q_j(\boldsymbol{\Theta})}\right)\left(\frac{12}{c_*}\right)^{d+1}\right),\nonumber
\end{gather}
for $c_*>0$, where 
\begin{gather}
\Upsilon_M({\bf 0})\equiv\lim_{{\bf t}\to{\bf 0}}\Upsilon_M({\bf t})=\frac{\hat{\Phi}_*(M)}{\Gamma(d/2+1)}\left(\frac{M}{2^{3/2}}\right)^d,\quad
\hat{\Phi}_*(M)=\inf_{\|\boldsymbol{\omega}\|_2\le 2M}\hat{\Phi}(\boldsymbol{\omega})\nonumber
\end{gather}
for
$M>0$,
and the respective \emph{local} and \emph{global} separation distances with respect to the Mahalanobis-like distance ${\rm d}_{\boldsymbol{\Theta}}$, $\boldsymbol{\Theta}$ non-singular are given by 
\begin{gather}
q_j(\boldsymbol{\Theta})=\frac{1}{2}\min_{k=1,\ldots,n,\;k\ne j}{\rm d}_{\boldsymbol{\Theta}}({\bf x}_j,{\bf x}_k)\quad{\rm and}\quad q(\boldsymbol{\Theta})=\min_j q_j(\boldsymbol{\Theta}).\nonumber
\end{gather}
\end{thm}
This result is now applied 
to the Stationary Model and Non-Stationary Model (Subsections \ref{nominalStationary} and \ref{nominalNonStationary}, respectively).
\subsection{Stationary Model}\label{numericStationary}

Here,
\begin{gather}
\begin{split}
\lambda_{\rm min}(\Psi_{\boldsymbol{\theta}}({\bf X},{\bf X}))&=\min_{\|{\bf a}\|_2=1}\sum_{i,j}a_i a_j\Psi_{\boldsymbol{\theta}}({\bf x}_i,{\bf x}_j)\\
&=\sigma^2\min_{\|{\bf a}\|_2=1}\sum_{i,j}a_i a_j\varphi(\|\boldsymbol{\Theta}({\bf x}_i-{\bf x}_j)\|_2)\\
&=\sigma^2\min_{\|{\bf a}\|_2=1}\sum_{i,j}a_i a_j\varphi(\|{\bf x}^*_i-{\bf x}^*_j\|_2)\\
&=\sigma^2\min_{\|{\bf a}\|_2=1}\sum_{i,j}a_i a_j\Phi({\bf x}^*_i-{\bf x}^*_j)\\
&\ge\sigma^2 \min_{\|{\bf a}\|_2=1}\sum_{i=1}^n a_i^2 \ell_i(\boldsymbol{\Theta}),\label{case1LowerBound}
\end{split}
\end{gather}
where $\Phi({\bf x}^*-{\bf y}^*)=\varphi(\|{\bf x}^*-{\bf y}^*\|_2)$ and
\begin{gather}
\ell_i(\boldsymbol{\Theta})=\Upsilon_{c_*/q(\boldsymbol{\Theta})}({\bf 0})\left(1-\frac{\Gamma^2(d/2+1)\pi}{18}\left(\frac{q(\boldsymbol{\Theta})}{q_i(\boldsymbol{\Theta})}\right)\left(\frac{12}{c_*}\right)^{d+1}\right).\label{elli}
\end{gather}
The lower bound (\ref{case1LowerBound}) is maximized for $\ell_i(\boldsymbol{\Theta})$ constant over $i$ and as large as possible. This implies $q_i(\boldsymbol{\Theta})=q(\boldsymbol{\Theta})$ for all $i$. Now, the lower bound depends on 
\begin{gather}
\Upsilon_{c_*/q(\boldsymbol{\Theta})}({\bf 0})=\frac{\hat{\Phi}_*(c_*/q(\boldsymbol{\Theta}))}{\Gamma(d/2+1)}\left(\frac{c_*/q(\boldsymbol{\Theta})}{2^{3/2}}\right)^d,\label{Upsilon0star}
\end{gather}
which, for sufficiently small $q(\boldsymbol{\Theta})$, is an increasing function of $q(\boldsymbol{\Theta})$ that approaches zero as $q(\boldsymbol{\Theta})$ approaches zero. That is, in the stationary situation,
numeric accuracy is preserved for designs which are well-separated. 

\subsection{Non-Stationary Model}\label{numericNonStationary}

Here, assume \emph{additionally} that
$\boldsymbol{\Theta}_2=a\boldsymbol{\Theta}_1$ for some $a>1$.
Slightly, coarsen the bounds by replacing $\Upsilon_M({\bf 0})$ with its monotone decreasing in $M$ lower bound $\tilde{\Upsilon}^0_M=\inf_{m\in[r_*,M]}\Upsilon_m({\bf 0})$,
$r_*=c_*/\max_{{\bf x},{\bf y}\in\Omega}{\rm d}_{\boldsymbol{\Theta}_1}({\bf x},{\bf y})$. 
Let $\tilde{\ell}_i$ denote the coarsened version of (\ref{elli}).
Then,
\begin{gather}
\begin{split}
&\lambda_{\rm min}(\Psi_{\boldsymbol{\theta}}({\bf X},{\bf X}))=\min_{\|{\bf a}\|_2=1}\sum_{i,j}a_i a_j\Psi_{\boldsymbol{\theta}}({\bf x}_i,{\bf x}_j)\\
&=\sigma^2\min_{\|{\bf a}\|_2=1}\sum_{i,j}a_i a_j\left(\omega_1({\bf x}_i)\omega_1({\bf x}_j)\varphi(\|\boldsymbol{\Theta}_1({\bf x}_i-{\bf x}_j)\|_2)+\omega_2({\bf x}_i)\omega_2({\bf x}_j)\varphi(\|\boldsymbol{\Theta}_2({\bf x}_i-{\bf x}_j)\|_2)\right)\\
&\ge\sigma^2\min_{\|{\bf a}\|_2=1}\sum_{i}^n a_i^2\left(\omega_1({\bf x}_i)^2 \tilde{\ell}_i(\boldsymbol{\Theta}_1)+\omega_2({\bf x}_i)^2 \tilde{\ell}_i(\boldsymbol{\Theta}_2)\right),\label{case2LowerBound}
\end{split}
\end{gather}
where 
$\hat{\Phi}$ is the Fourier transform of $\Phi$ defined by $\Phi({\bf x}^*-{\bf y}^*)=\varphi(\|{\bf x}^*-{\bf y}^*\|_2)$ in (\ref{elli}) and (\ref{Upsilon0star}).
The lower bound (\ref{case2LowerBound}) is maximized for $\omega_1({\bf x}_i)^2 \tilde{\ell}_i(\boldsymbol{\Theta}_1)+\omega_2({\bf x}_i)^2 \tilde{\ell}_i(\boldsymbol{\Theta}_2)$ constant over $i$ and as large as possible.

Consider two design points ${\bf x}_i$ and ${\bf x}_j$ and suppose that the points in the input space \emph{near} ${\bf x}_i$ have more weight on the global, long range, correlation than the points in the input space \emph{near} ${\bf x}_j$ and the points in the input space \emph{near} ${\bf x}_j$ have more weight on the local, short range, correlation than the points in the input space \emph{near} ${\bf x}_i$, in the sense that
\begin{gather}
\begin{split}
\omega_1({\bf x}_i)^2(\tilde{\ell}_i(\boldsymbol{\Theta}_2)-\tilde{\ell}_i(\boldsymbol{\Theta}_1))&\ge\omega_1({\bf x}_j)^2(\tilde{\ell}_j(\boldsymbol{\Theta}_2)-\tilde{\ell}_j(\boldsymbol{\Theta}_1)),\\
\omega_2({\bf x}_i)^2(\tilde{\ell}_i(\boldsymbol{\Theta}_2)-\tilde{\ell}_i(\boldsymbol{\Theta}_1))&\le\omega_2({\bf x}_j)^2(\tilde{\ell}_j(\boldsymbol{\Theta}_2)-\tilde{\ell}_j(\boldsymbol{\Theta}_1)).\label{nonstationarityAssumptionNumeric}
\end{split}
\end{gather}
Here, we consider the situation where $q_i(\boldsymbol{\Theta}_1)$ and $q_i(\boldsymbol{\Theta}_2)$ are \emph{small} across $i$, the situation where a bound on the numeric error is most relevant.
For $q(\boldsymbol{\Theta}_1)$ and $q(\boldsymbol{\Theta}_2)$ sufficiently small, $\tilde{\Upsilon}^0_{c_*/q(\boldsymbol{\Theta})}$ is \emph{strictly} increasing in $q$.
Further, the assumption $\boldsymbol{\Theta}_2=a\boldsymbol{\Theta}_1$ for some $a>1$ implies $q(\boldsymbol{\Theta}_1)/q_i(\boldsymbol{\Theta}_1)= q(\boldsymbol{\Theta}_2)/q_i(\boldsymbol{\Theta}_2)$ and $q(\boldsymbol{\Theta}_2)> q(\boldsymbol{\Theta}_1)$. 
Together, these facts imply $\tilde{\ell}_i(\boldsymbol{\Theta}_2)>\tilde{\ell}_i(\boldsymbol{\Theta}_1)$.
Uniformity of the bounds (\ref{case2LowerBound}) along with $\omega_1(\cdot)^2+\omega_2(\cdot)^2=1$ gives
\begin{gather}
\begin{split}
&\omega_1({\bf x}_i)^2(\tilde{\ell}_i(\boldsymbol{\Theta}_1)-\tilde{\ell}_i(\boldsymbol{\Theta}_2))-\omega_1({\bf x}_j)^2(\tilde{\ell}_j(\boldsymbol{\Theta}_1)-\tilde{\ell}_j(\boldsymbol{\Theta}_2))=\tilde{\ell}_j(\boldsymbol{\Theta}_2) - \tilde{\ell}_i(\boldsymbol{\Theta}_2),\\
&\omega_2({\bf x}_j)^2(\tilde{\ell}_j(\boldsymbol{\Theta}_1)-\tilde{\ell}_j(\boldsymbol{\Theta}_2))-\omega_2({\bf x}_i)^2(\tilde{\ell}_i(\boldsymbol{\Theta}_1)-\tilde{\ell}_i(\boldsymbol{\Theta}_2))=\tilde{\ell}_j(\boldsymbol{\Theta}_1)- \tilde{\ell}_i(\boldsymbol{\Theta}_1).\label{uniformityNumeric}
\end{split} 
\end{gather}
Combining (\ref{nonstationarityAssumptionNumeric}) and (\ref{uniformityNumeric}) with the fact that $\tilde{\ell}_i(\boldsymbol{\Theta})$ is an increasing function of $q_i(\boldsymbol{\Theta})$ for small $q_i(\boldsymbol{\Theta})$ gives
\begin{gather}
\begin{split}
q_j(\boldsymbol{\Theta}_1)<q_i(\boldsymbol{\Theta}_1)\quad{\rm and}\quad q_j(\boldsymbol{\Theta}_2)<q_i(\boldsymbol{\Theta}_2).\nonumber
\end{split}
\end{gather}
That is, a uniform bound on (\ref{case2LowerBound}) is achieved by an experimental design ${\bf X}$ which has smaller separation distance, with respect to either ${\rm d}_{\boldsymbol{\Theta}_1}$ or ${\rm d}_{\boldsymbol{\Theta}_2}$, in regions with more emphasis on the local, more quickly decaying, correlation and less emphasis on the global, more slowly decaying, correlation. Note that in the numeric accuracy context, the global and local \emph{emphases}, for \emph{small} $q_i(\boldsymbol{\Theta}_1)$ and $q_i(\boldsymbol{\Theta}_2)$, at ${\bf x}_i$ are given concretely by 
$\omega_1({\bf x}_i)^2(\tilde{\ell}_i(\boldsymbol{\Theta}_2)-\tilde{\ell}_i(\boldsymbol{\Theta}_1))$ and 
$\omega_2({\bf x}_i)^2(\tilde{\ell}_i(\boldsymbol{\Theta}_2)-\tilde{\ell}_i(\boldsymbol{\Theta}_1))$, respectively.


Example high quality 23 run experimental designs for the \emph{numeric} situations described in 
Subsection \ref{numericStationary} (Stationary Model) and Subsection \ref{numericNonStationary} (Non-Stationary Model)
are illustrated in the left and right panels, respectively, of Figure \ref{NumDesign}. 
For the Stationary Model example, $\varphi(d)={\rm exp}\{-d^2\}$, the so-called \emph{Gaussian} correlation function.
For the Non-Stationary Model example,
$\varphi(d)$ is Wendland's kernel with $k=10$ \cite{wendland2005}, 
$\omega_1({\bf u})^2=1-\|{\bf u}\|_2^2/2$, $\omega_2({\bf u})^2=\|{\bf u}\|_2^2/2$,
$\boldsymbol{\Theta}_1=0.1\cdot{\bf I}_2$, and $\boldsymbol{\Theta}_2=1\cdot{\bf I}_2$.
For both cases,
$\sigma^2=1$.
As expected, in the first panel, illustrating the stationary situation, the design points lie near a triangular lattice (subject to edge effects), similar to, but expanded towards the edges of the design space relative to, the nominal design.
Similarly, in the second panel, illustrating the non-stationary correlation situation, the design points in the upper right-hand corner, where the shorter range, more quickly decaying, correlation is emphasized, are more dense than in the lower left-hand corner, where the longer range, more slowly decaying, correlation is emphasized.
While the provided bounds hold for all $c _*>0$, the actual value of the bounds depends on the selected value of $c_*$.
Here, we take $c_*=1.1\times 12\left(\frac{18}{\pi\Gamma^2(d/2+1)}\right)^{-1/(d+1)}$.
Similarly to the nominal examples, the optimization routine was initialized at a triangular lattice, scaled to maximize the separation distance. For the stationary situation, a homotopy continuation \cite{homotopy} approach
was used,
slowly transitioning from maximizing the mean of the local bounds ${\ell}_i(\boldsymbol{\Theta})$ to maximizing the minimum of the local bounds.
For the non-stationary situation, a homotopy continuation approach with two stages was used, first transitioning from $\omega_1,\omega_2$ constant to varying over the input space, then transitioning from maximizing the mean of the local bounds $\omega_1({\bf x}_i)^2\tilde{\ell}_i(\boldsymbol{\Theta}_1)+\omega_2({\bf x}_i)^2\tilde{\ell}_i(\boldsymbol{\Theta}_2)$ to maximizing the minimum of the local bounds.
Nelder-Mead \cite{nocedal} black box optimization was used throughout. 
\begin{figure}[!h]
\centering
\subfigure{\includegraphics[width=1.65in]{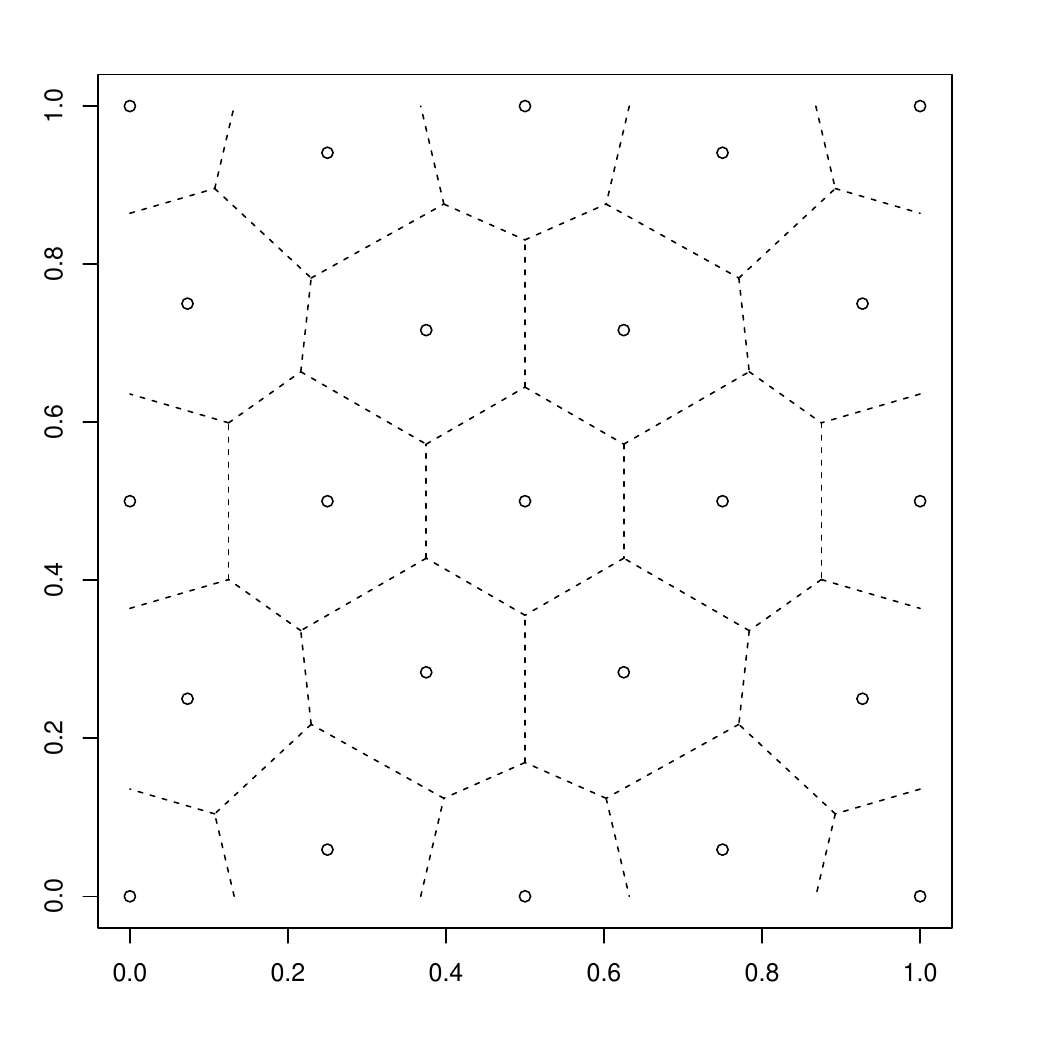}}
\subfigure{\includegraphics[width=1.65in]{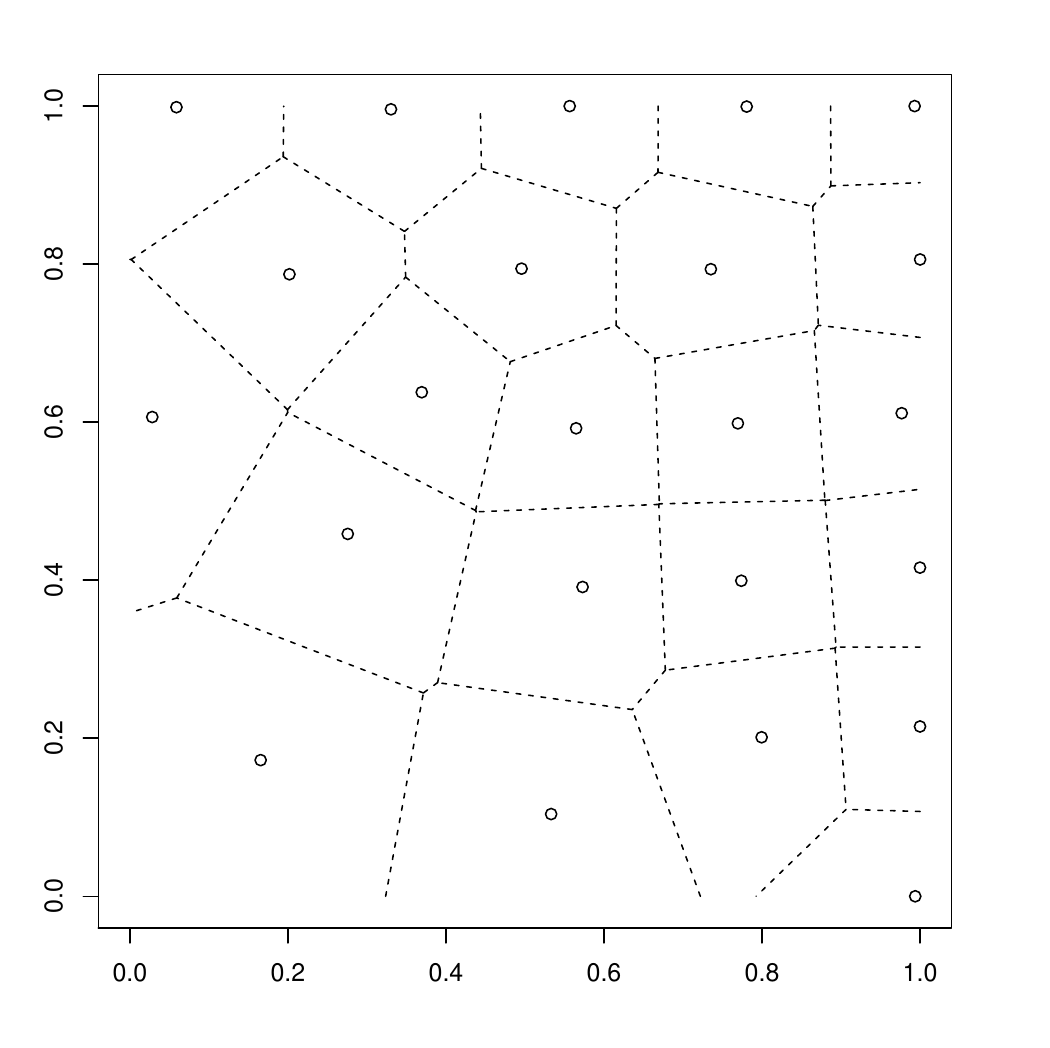}}
\caption{{\bf Left Panel:} Numeric error design for stationary correlation. {\bf Right Panel:} Numeric error design for the Case 2 model of non-stationary correlation with $\varphi(d)$ Wendland's kernel with $k=10$, 
$\omega_1({\bf u})^2=1-\|{\bf u}\|_2^2/2$, $\omega_2({\bf u})^2=\|{\bf u}\|_2^2/2$,
$\boldsymbol{\Theta}_1=0.1\cdot{\bf I}_2$, and $\boldsymbol{\Theta}_2=1\cdot{\bf I}_2$}
\label{NumDesign}
\end{figure}

\section{Parameter Estimation}\label{paramEst}

Consider maximum likelihood estimation and let $\mathbb{E}$ denote the expectation conditional on ${\bf X}$ and $f({\bf X})$.
Then, for $n$ not too small,
\begin{gather} 
\begin{split}
\mathbb{E}\left\{\hat{f}_{\boldsymbol{\vartheta}_*}({\bf x})-\hat{f}_{\hat{\boldsymbol{\vartheta}}}({\bf x})\right\}^2&\approx\frac{\partial\hat{f}_{\boldsymbol{\vartheta}_*}({\bf x})}{\partial\boldsymbol{\vartheta}'_*}{\rm Var}\;\hat{\boldsymbol{\vartheta}}\frac{\partial\hat{f}_{\boldsymbol{\vartheta}_*}({\bf x})}{\partial\boldsymbol{\vartheta}_*}\\
&\approx \frac{\partial\hat{f}_{\boldsymbol{\vartheta}_*}({\bf x})}{\partial\boldsymbol{\vartheta}'_*}\mathcal{I}({\boldsymbol{\vartheta}}_*)^{-1} \frac{\partial\hat{f}_{\boldsymbol{\vartheta}_*}({\bf x})}{\partial\boldsymbol{\vartheta}_*}
,\label{approxExpNumErr}
\end{split}
\end{gather}
where $\mathcal{I}({\boldsymbol{\vartheta}}_*)=\mathbb{E}\frac{\partial\ell}{\partial\boldsymbol{\vartheta}_*}\frac{\partial\ell}{\partial\boldsymbol{\vartheta}'_*}$ denotes the information matrix
and $\ell$ denotes the log-likelihood of the data $f({\bf X})$.
Roughly, a high-quality design for parameter estimation will have $\left\|\frac{\partial\hat{f}_{\boldsymbol{\vartheta}_*}({\bf x})}{\partial\boldsymbol{\vartheta}_*}\right\|_2$ small and $\lambda_{\rm min}(\mathcal{I}({\boldsymbol{\vartheta}}_*))$ large.
Arrange the vector of parameters as $\boldsymbol{\vartheta}= \begin{pmatrix}\boldsymbol{\beta}' & \boldsymbol{\theta}'\end{pmatrix}'$ and $ \boldsymbol{\theta} = \begin{pmatrix} \sigma^2 & \boldsymbol{\varrho}'\end{pmatrix}'$.
Throughout the parameter estimation section, take 
\begin{gather}
\Psi_{\boldsymbol{\theta}}({\bf u,v})=\sigma^2\Phi_{\boldsymbol{\varrho}}({\bf u},{\bf v})\label{paramEstAssumption}.
\end{gather}
Expressions for the components of the right-hand side of (\ref{approxExpNumErr}) are provided in Lemma \ref{paramProp} in Appendix \ref{paramThmProof}. 
These expressions are in turn used to develop the approximate upper bound for the mean squared prediction error given in Theorem \ref{paramThm}.
Proofs of Lemma \ref{paramProp} and Theorem \ref{paramThm} are provided in Appendix \ref{paramThmProof}.
\begin{thm}\label{paramThm}
Suppose $f(\cdot)\sim{\rm GP}({\bf h}(\cdot)'\boldsymbol{\beta},\sigma^2\Phi_{\boldsymbol{\varrho}}(\cdot,\cdot))$ for fixed, known regression functions ${\bf h}(\cdot)$ and positive definite $\Phi_{\boldsymbol{\varrho}}(\cdot,\cdot)$.
Further, assume that the input locations ${\bf X}$ have large sample distribution $F$.
Let $\hat{\boldsymbol{\vartheta}}$ denote the maximum likelihood estimator of the unknown parameters $\boldsymbol{\vartheta}= \begin{pmatrix}\boldsymbol{\beta}' & \sigma^2 & \boldsymbol{\varrho}'\end{pmatrix}'$.
Then, an approximate upper bound for\\ $\mathbb{E}\left\{\hat{f}_{\boldsymbol{\vartheta}_*}({\bf x})-\hat{f}_{\hat{\boldsymbol{\vartheta}}}({\bf x})\right\}^2$ is given by
\begin{gather}
\begin{split}
\sup_{{\bf u},{\bf v}\in\Omega}\Phi_{\boldsymbol{\varrho}}({\bf u},{\bf v})\left(\frac{\|{\bf c}_1\|_2^2}{s_1}+\frac{2\sup_{{\bf u},{\bf v}\in\Omega}\Phi_{\boldsymbol{\varrho}}({\bf u},{\bf v})\|{\bf c}_3\|_2^2}{s_2}\right)
\nonumber
\end{split}
\end{gather}
where 
\begin{gather}
\begin{split}
&{\bf c}_1 = {\bf h(x)} - {\bf H(X)}' \Psi_{\boldsymbol{\theta}}({\bf X,X})^{-1} \Psi_{\boldsymbol{\theta}}({\bf X,x}),\\
&{\bf c}_3=\left(\frac{\partial\Psi_{\boldsymbol{\theta}}({\bf x,X})}{\partial\boldsymbol{\varrho}} - ({\bf I}_d\otimes\Psi_{\boldsymbol{\theta}}({\bf x,X})\Psi_{\boldsymbol{\theta}}({\bf X,X})^{-1})\frac{\partial\Psi_{\boldsymbol{\theta}}({\bf X,X})}{\partial\boldsymbol{\varrho}}\right)\Psi_{\boldsymbol{\theta}}({\bf X,X})^{-1}\delta({\bf X}),\nonumber
\end{split}
\end{gather}
$s_1=\lambda_{\rm min}\left(\int {\bf h}({\bf y}){\bf h}({\bf y})'{\rm d}F({\bf y})\right)$,
and $s_2$ is implicitly defined in (\ref{boundStar2}).
In particular, $s_1>0$ 
unless ${\bf h}({\bf y})'{\bf a}=0$ with probability $1$ with respect to $F$ for some ${\bf a}\ne 0$
and 
$s_2>0$ unless $\frac{\partial\Phi_{\boldsymbol{\varrho}}({\bf x},{\bf y})}{\partial\boldsymbol{\varrho}'}{\bf a}=\Phi_{\boldsymbol{\varrho}}({\bf x},{\bf y})b$ with probability 1 with respect to $F\times F$ for some $\left({\bf a}'\;\;b\right)'\ne {\bf 0}$.
%
\end{thm}
This upper bound is approximate in the sense that for a sequence of experimental designs for which the maximum likelihood estimates converge, the probability that the upper bound is violated by more than $\varepsilon>0$ goes to zero.

The term $\|{\bf c}_3\|_2^2$ admits the simple upper bound
\begin{gather}
\left\|\frac{\partial\Psi_{\boldsymbol{\theta}}({\bf x,X})}{\partial\boldsymbol{\varrho}} - ({\bf I}_u\otimes\Psi_{\boldsymbol{\theta}}({\bf x,X})\Psi_{\boldsymbol{\theta}}({\bf X,X})^{-1})\frac{\partial\Psi_{\boldsymbol{\theta}}({\bf X,X})}{\partial\boldsymbol{\varrho}}\right\|_2^2\left\|\delta({\bf X})\right\|_2^2\left/\lambda_{\rm min}\left(\Psi_{\boldsymbol{\theta}}({\bf X,X})\right)^2.\right.\nonumber
\end{gather}
The term $\left\|\delta({\bf X})\right\|_2=\left\|f({\bf X})-{\bf H(X)}{\boldsymbol{\beta}}\right\|_2$ is an approximation to $\sqrt{n}\left\|f(\cdot)-{\bf h(\cdot)}'{\boldsymbol{\beta}}\right\|_{L_2(F)}$ with respect to the large sample distribution of the data $F$.
Further, 
$\lambda_{\rm min}\left(\Psi_{\boldsymbol{\theta}}({\bf X,X})\right)$
is well-controlled by experimental designs which maintain high-quality numeric properties.
Similarly, both ${\bf h}({\bf x})-{\bf H}({\bf X})'\Psi_{\boldsymbol{\theta}}({\bf X},{\bf X})^{-1}\Psi_{\boldsymbol{\theta}}({\bf X},{\bf x})$ and\\ 
$\frac{\partial\Psi_{\boldsymbol{\theta}}({\bf x,X})}{\partial\boldsymbol{\varrho}} - ({\bf I}_u\otimes\Psi_{\boldsymbol{\theta}}({\bf x,X})\Psi_{\boldsymbol{\theta}}({\bf X,X})^{-1})\frac{\partial\Psi_{\boldsymbol{\theta}}({\bf X,X})}{\partial\boldsymbol{\varrho}}$ are nominal interpolation errors, respectively for the regression functions and (the transpose of) the Jacobian of $\Psi_{\boldsymbol{\theta}}({\bf X,x})$ with respect to the correlation parameters.
As discussed towards the end of Section \ref{nominal}, we expect the norms of both of these interpolation errors to behave in a manner similar to Gaussian process or RKHS interpolation. That is, the norms of both of these terms will be small for experimental designs which are high-quality with respect to nominal error. 

As discussed towards the end of Section \ref{nominal}, $s_1$ will be large for sets of input locations which have good traditional experimental design properties.
The term $s_2$ will be large for experimental designs whose system of differences $\{{\bf x}_i-{\bf x}_j\}$ make $\frac{\partial\Phi_{\boldsymbol{\varrho}}({\bf x}_i,{\bf x}_j)}{\partial\boldsymbol{\varrho}}$ far from zero, \emph{balanced} with respect to a basis of $\mathbb{R}^{{\rm dim}\;\boldsymbol{\varrho}}$, and \emph{not collinear} with  $\Phi_{\boldsymbol{\varrho}}({\bf x}_i,{\bf x}_j)$. 
Consider as an example,
underlying kernels which \emph{depend only on the difference between their arguments} and are \emph{radially decreasing} in the sense that $\Phi(\boldsymbol{\delta}_1)\ge\Phi(\boldsymbol{\delta}_2)$ if $\|\boldsymbol{\delta}_1\|_2\le\|\boldsymbol{\delta}_2\|_2$ with $\Phi_{\boldsymbol{\varrho}}(\cdot)=\Phi({\rm diag}\{\boldsymbol{\varrho}\}(\cdot))$.
For radially decreasing underlying kernels $\Phi$, the term $\frac{\partial\Phi_{\boldsymbol{\varrho}}({\bf x}_i-{\bf x}_j)}{\partial\boldsymbol{\varrho}}$ is near zero if ${\bf x}_i-{\bf x}_j$ is near zero or far from zero, while the term $\frac{\partial\Phi_{\boldsymbol{\varrho}}({\bf x}_i-{\bf x}_j)}{\partial\boldsymbol{\varrho}}$ has negative components if the difference ${\bf x}_i-{\bf x}_j$ is slightly beyond the location where $\Phi({\rm diag}\{\boldsymbol{\theta}\}(\cdot))$ is decreasing most rapidly along each coordinate axis, since $\frac{\partial\Phi_{\boldsymbol{\varrho}}({\bf x}_i-{\bf x}_j)}{\partial\boldsymbol{\varrho}}={\rm diag}\{{\bf x}_i-{\bf x}_j\}\nabla\Phi({\rm diag}\{\boldsymbol{\varrho}\}({\bf x}_i-{\bf x}_j))$. 
In this situation, $\frac{\partial\Phi_{\boldsymbol{\varrho}}({\bf x}_i-{\bf x}_j)}{\partial\boldsymbol{\varrho}}$ has negative components and $\Phi_{\boldsymbol{\varrho}}({\bf x}_i-{\bf x}_j)$ is a non-negative weighting function, 
so they could only be (nearly) collinear for experimental designs which make almost all $\frac{\partial\Phi_{\boldsymbol{\varrho}}({\bf x}_i-{\bf x}_j)}{\partial\boldsymbol{\varrho}}$ near zero.
Figure \ref{cijPlots} shows $\Phi_{\boldsymbol{\varrho}}(\cdot)$ and both components of $\frac{\partial\Phi_{\boldsymbol{\varrho}}(\cdot)}{\partial\boldsymbol{\varrho}}$ 
for $\Phi({\bf d})={\rm exp}\{-{\bf d}'{\bf d}\}$ and $\boldsymbol{\varrho}=\begin{pmatrix}1&2\end{pmatrix}'$.
Pairs of points ${\bf x}_i,\;{\bf x}_j$ whose difference lies slightly beyond the location where $\Phi({\rm diag}\{\boldsymbol{\theta}\}(\cdot))$ is decreasing most rapidly along each coordinate axis have potential to increase eigenvalues of 
$\sum_{ij}\frac{\partial\Phi_{\boldsymbol{\varrho}}({\bf x}_i-{\bf x}_j)}{\partial\boldsymbol{\varrho}}\frac{\partial\Phi_{\boldsymbol{\varrho}}({\bf x}_i-{\bf x}_j)}{\partial\boldsymbol{\varrho}'}$.
Further, 
$\lambda_{\rm min}\left(\sum_{ij}\frac{\partial\Phi_{\boldsymbol{\varrho}}({\bf x}_i-{\bf x}_j)}{\partial\boldsymbol{\varrho}}\frac{\partial\Phi_{\boldsymbol{\varrho}}({\bf x}_i-{\bf x}_j)}{\partial\boldsymbol{\varrho}'}\right)$ is large
for sets of differences $\{{\bf x}_i-{\bf x}_j\}$ which \emph{balance} the differences along coordinate axes in the sense that
\begin{gather}
n_k \max_{\bf d}\left\{\frac{\partial\Phi_{\boldsymbol{\varrho}}({\bf d})}{\partial\boldsymbol{\varrho}}\right\}_k^2\approx n_l \max_{\bf d}\left\{\frac{\partial\Phi_{\boldsymbol{\varrho}}({\bf d})}{\partial\boldsymbol{\varrho}}\right\}_l^2,\nonumber
\end{gather}
for $k,l=1,\ldots,d$ where $\left\{\cdot\right\}_k$ denotes element $k$ of 
its argument
and $n_k$ denotes the number of differences (of length slightly beyond the location where $\Phi({\rm diag}\{\boldsymbol{\theta}\}(\cdot))$ is decreasing most rapidly) along coordinate axis $k$. In the example described above and illustrated in Figure \ref{cijPlots}, $\left\{\nabla\Phi({\rm diag}\{\boldsymbol{\varrho}\}({\bf d}))'{\rm diag}\{{\bf d}\}\right\}^2_1\approx (0.8)^2$ and $\left\{\nabla\Phi({\rm diag}\{\boldsymbol{\varrho}\}({\bf d}))'{\rm diag}\{{\bf d}\}\right\}^2_2\approx (0.4)^2$, respectively at $\begin{pmatrix}\pm 1&0\end{pmatrix}'$ and $\begin{pmatrix}0&\pm 0.5\end{pmatrix}'$, so an experimental design \emph{solely targeting the eigenvalues of} 
$\sum_{ij}\frac{\partial\Phi_{\boldsymbol{\varrho}}({\bf x}_i-{\bf x}_j)}{\partial\boldsymbol{\varrho}}\frac{\partial\Phi_{\boldsymbol{\varrho}}({\bf x}_i-{\bf x}_j)}{\partial\boldsymbol{\varrho}'}$
would have roughly $n_1$ differences ${\bf x}_i-{\bf x}_j=\begin{pmatrix}\pm 1&0\end{pmatrix}'$ and $n_2$ differences ${\bf x}_i-{\bf x}_j=\begin{pmatrix}0&\pm 0.5\end{pmatrix}'$ where
\begin{gather}
n_1(0.8)^2=n_2(0.4)^2\Longrightarrow n_1=\frac{n_2}{4}.\nonumber
\end{gather}

\begin{figure}
\centering
\includegraphics[width=4.95in]{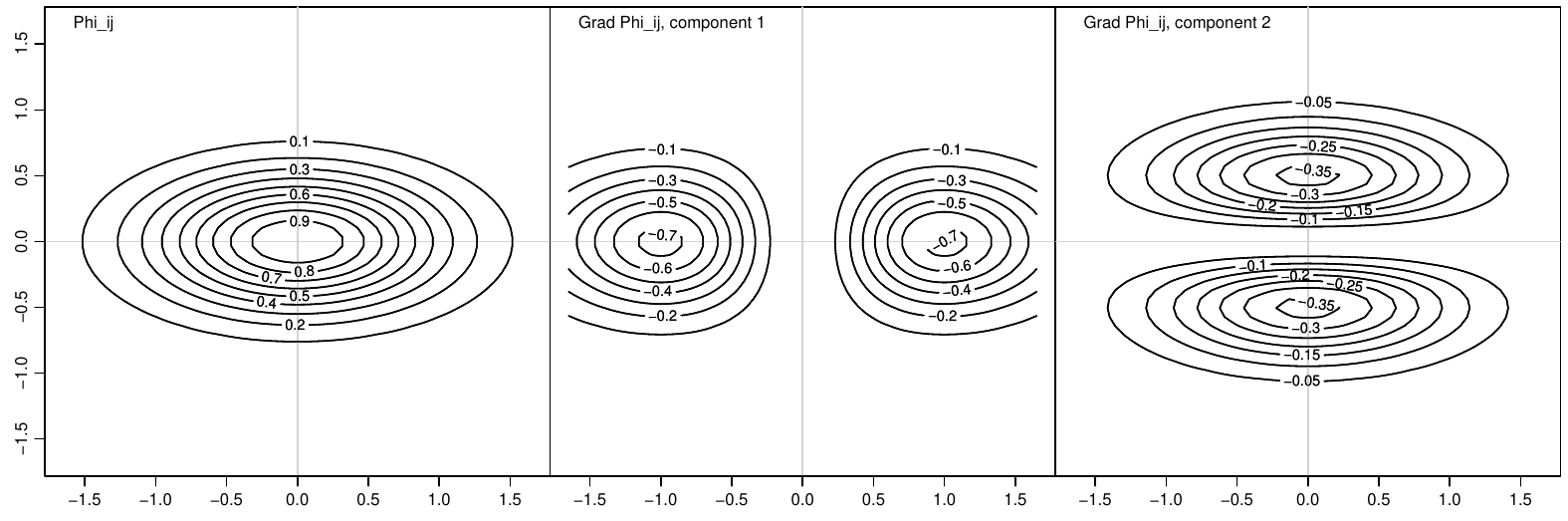}
\caption{
{\bf Left Panel:} $\Phi_{\boldsymbol{\varrho}}(\cdot)$.
{\bf Middle and Right Panels:} Components of $\frac{\partial\Phi_{\boldsymbol{\varrho}}(\cdot)}{\partial\boldsymbol{\varrho}}$. 
}\label{cijPlots}
\end{figure}

Consider another example with $\Phi({\bf d})={\rm exp}\{-{\bf d}'{\bf d}\}$, $\Phi_{\boldsymbol{\varrho}}(\cdot)=\Phi({\rm diag}\{\boldsymbol{\varrho}\}(\cdot))$, and $\boldsymbol{\varrho}=\left(\;3\quad 3\;\right)'$. An experimental design maximizing $\lambda_{\rm min}\left(\sum_{ij}\frac{\partial\Phi_{\boldsymbol{\varrho}}({\bf x}_i-{\bf x}_j)}{\partial\boldsymbol{\varrho}}\frac{\partial\Phi_{\boldsymbol{\varrho}}({\bf x}_i-{\bf x}_j)}{\partial\boldsymbol{\varrho}'}\right)$ 
is shown in the left panel of Figure \ref{ParamDesign}. There are 11 points at the middle location and 3 at each peripheral location. In particular, 
this design 
is \emph{not} space-filling.
A high quality experimental design with respect to the upper-bound in Theorem \ref{paramThm} is shown in the right panel of Figure \ref{ParamDesign}.
The influence of 
$\lambda_{\rm min}\left(\sum_{ij}\frac{\partial\Phi_{\boldsymbol{\varrho}}({\bf x}_i-{\bf x}_j)}{\partial\boldsymbol{\varrho}}\frac{\partial\Phi_{\boldsymbol{\varrho}}({\bf x}_i-{\bf x}_j)}{\partial\boldsymbol{\varrho}'}\right)$
is substantially less than the influence of the space-filling properties controlling the nominal and numeric error.

\begin{figure}
\centering
\subfigure{\includegraphics[width=1.65in]{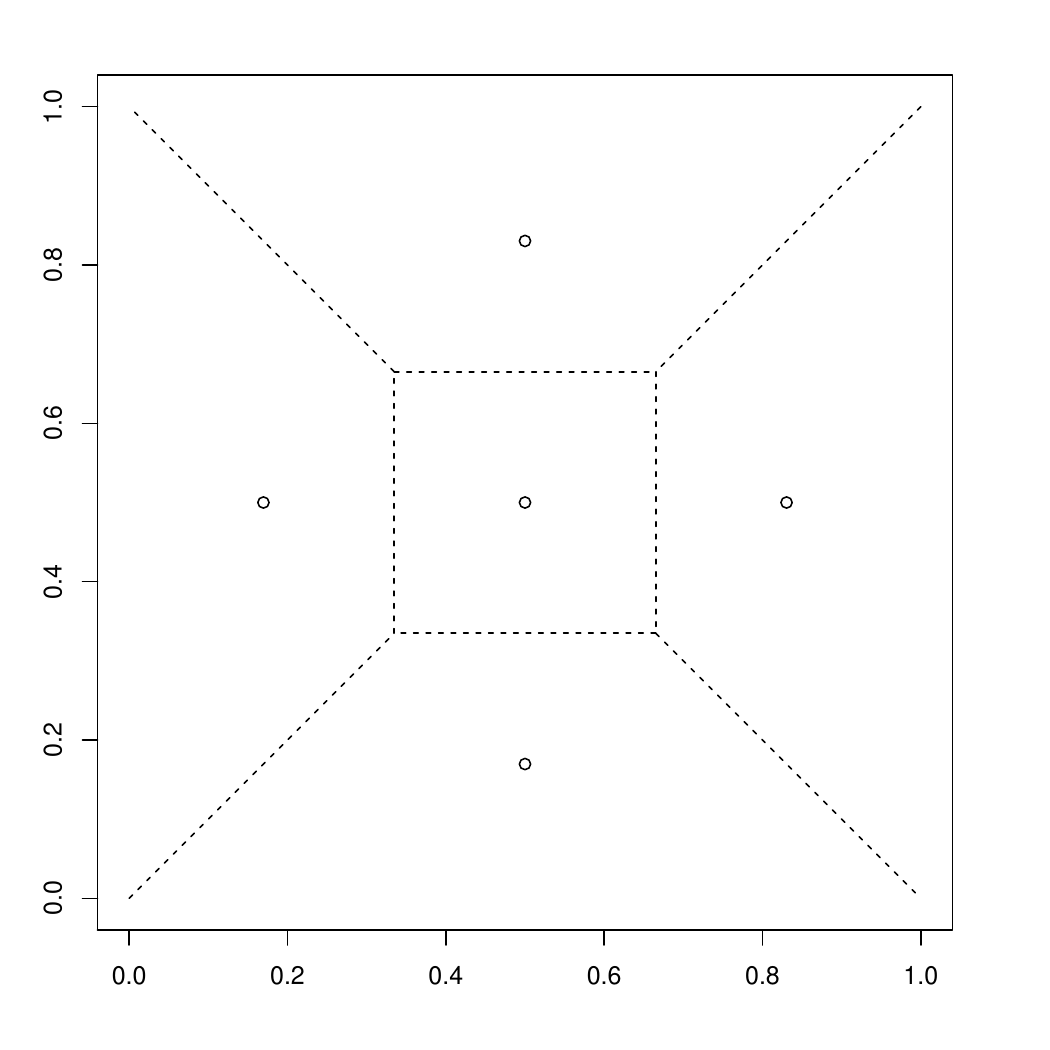}}
\subfigure{\includegraphics[width=1.65in]{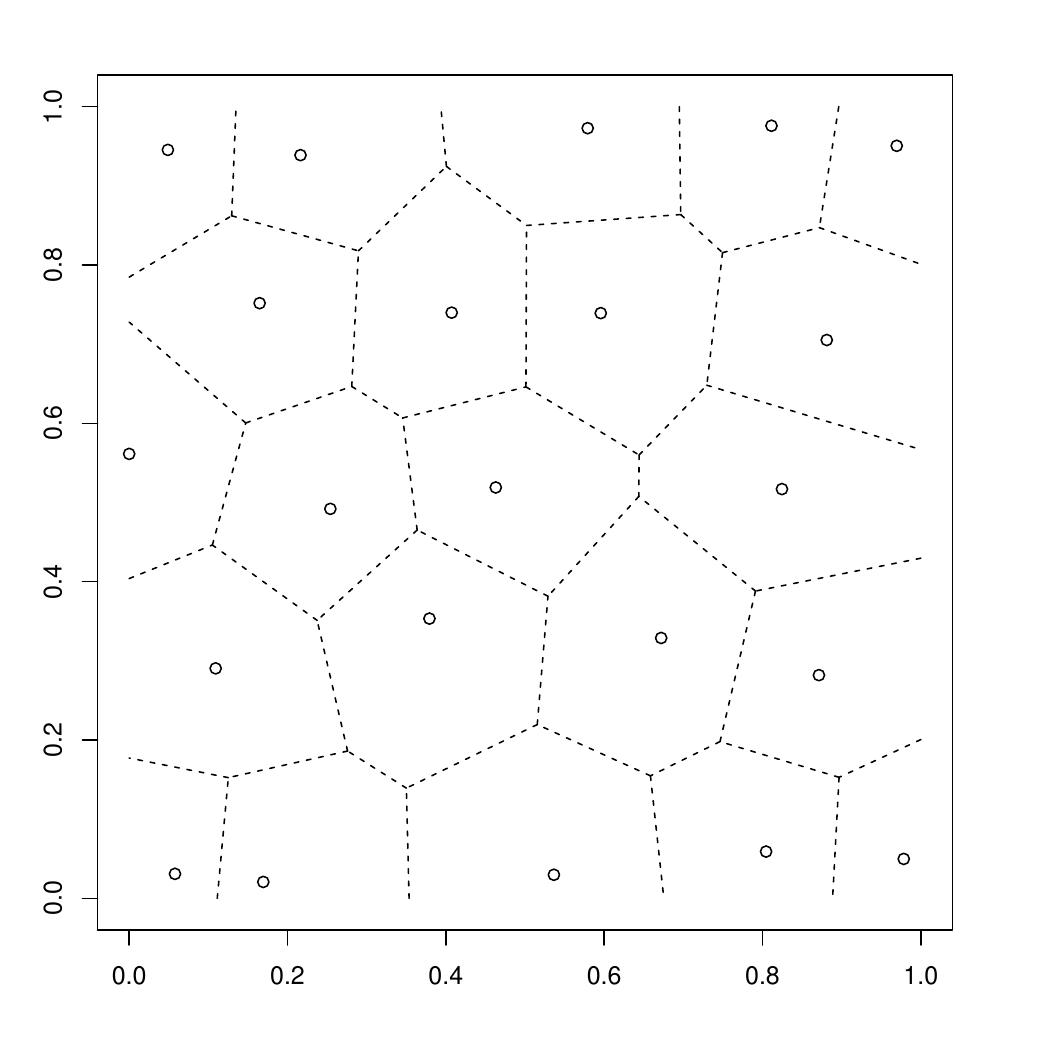}}
\caption{{\bf Left Panel:} Experimental design maximizing $\lambda_{\rm min}({\bf C}'_{\boldsymbol{\theta}}{\bf C}_{\boldsymbol{\theta}})$ and minimizing $\kappa({\bf C}'_{\boldsymbol{\theta}}{\bf C}_{\boldsymbol{\theta}})$ for $\Phi({\bf d})={\rm exp}\{-{\bf d}'{\bf d}\}$ and $\boldsymbol{\varrho}=\left(\;3\quad 3\;\right)'$. Note that there are 11 points at the middle location and 3 at each peripheral location. {\bf Right Panel:} Parameter estimation error design with respect to the upper-bound in Theorem \ref{paramThm}.}
\label{ParamDesign}
\end{figure}

\section{Discussion}\label{discussion}

Broadly applicable and rigorously justified principles of experimental design for Gaussian process emulation of deterministic computer experiments have been developed. 
The \emph{space-filling} properties ``small fill distance'' and ``large separation distance'', potentially with respect to an input space rescaling to accommodate varying rates of correlation decay depending on displacement orientation, are only weakly conflicting and ensure \emph{well-controlled} nominal, numeric, and parameter estimation error. 
The presence of non-stationarity in correlation requires a higher density of input locations in regions with more emphasis on the local, more quickly decaying, correlation, relative to input locations in regions with more emphasis on the global, more slowly decaying, correlation. Designs of this type can potentially be constructed via
homotopy continuation techniques \cite{homotopy}.
Alternatively, techniques in \cite{bowman} or \cite{joseph} may be useful for constructing high-quality experimental designs in a non-stationary setting.
The inclusion of regression functions results in high quality designs which balance traditional experimental design properties, targeting the variance of regression function coeffecients, with space-filling properties, while consideration of error in parameter estimation results in high quality designs slightly favoring pairs of input locations having particular lengths and orientations of their displacement vector.
The influence on the accuracy of emulation of regression functions and error in parameter estimation appears to be substantially less than the influence of the \emph{space-filling} properties ``small fill distance'' and ``large separation distance''.

The results presented in Theorems \ref{nominalTheorem}, \ref{numericTheorem}, \ref{paramThm}, and their subsequent discussions are generally well-aligned with \emph{distance-based} design criteria, such as minimax (minimize the maximum distance from inputs of interest to the design) and maximin (maximize the minimum distance between design points).
On their surface, the results appear to be somewhat less well-aligned with designs which emphasize \emph{low-dimensional projections}, such as Latin hypercube \cite{mckay}, orthogonal array-based Latin hypercube \cite{tang}, or MaxPro \cite{maxpro}. 
On the other hand, 
when some inputs to the unknown target function are inert, having little or no impact on the response, 
the presented results indicate that the corresponding correlation parameters (here described as inverse bandwidths) are near zero and high-quality designs should have space-fillingness in the relevant lower-dimensional projections.
Throughout this work, it has been tacitly assumed that a relatively accurate \emph{a priori} guess for the correlation parameters is available. 
Explicitly handling of \emph{a priori} uncertainty in correlation parameters could form an important extension of this work. 
In the case where one expects several inert variables, a design with good projection properties might reasonably be expected.

As a brief exploration, 
several of the proposed \emph{high-quality} designs were compared to one another and a spectrum of designs from the literature, in stationary, non-stationary, non-trivial regression functions, and inert inputs settings. In particular, five examples were considered, with results presented in Table 1. The \emph{Stationary Example} follows the setting illustrated in the left panel of Figure \ref{Design}, $\Psi({\bf u},{\bf v})=\exp\{-\|{\bf u}-{\bf v}\|_2^2\}$. 
The \emph{Non-Stationary Nominal Example} follows the
non-stationary example 
illustrated in the middle panel of Figure \ref{Design}, $\varphi(d)={\rm exp}\{-d^2\}$, 
$\omega_1({\bf u})^2=1-\|{\bf u}\|^2/2$, $\omega_2({\bf u})^2=\|{\bf u}\|^2/2$,
$\boldsymbol{\Theta}_1=1\cdot{\bf I}_2$, $\boldsymbol{\Theta}_2=10\cdot{\bf I}_2$.
The \emph{Regression Functions Example} follows the setting illustrated in the right panel of Figure \ref{Design}, $\Psi({\bf u},{\bf v})=\exp\{-\|{\bf u}-{\bf v}\|_2^2\}$ and a linear regression function for each dimension. 
The \emph{Non-Stationary Numeric Example} follows the 
non-stationary example illustrated in the right panel of Figure \ref{NumDesign}, $\varphi(d)$ Wendland's kernel with $k=10$, 
$\omega_1({\bf u})^2=1-\|{\bf u}\|^2/2$, $\omega_2({\bf u})^2=\|{\bf u}\|^2/2$,
$\boldsymbol{\Theta}_1=0.1\cdot{\bf I}_2$, $\boldsymbol{\Theta}_2=1\cdot{\bf I}_2$.
The \emph{${\bf x}_2$ Inert Example} follows a setting where the unknown function does not depend on ${\bf x}_2$, $\Psi({\bf u},{\bf v})=\exp\{-({u}_1-{v}_2)^2\}$.
Compared designs include \emph{Nominal}, \emph{Nominal Non-Stationary}, and \emph{Regression} shown respectively in the left, middle, and right panels of Figure \ref{Design}, \emph{Numeric} and \emph{Numeric Non-Stationary} shown respectively in the left and right panels of Figure \ref{NumDesign}, and \emph{Parameter Estimation} shown in the right panel of Figure \ref{ParamDesign}.
Designs from the literature include 
random uniform design points,
random Latin hypercube, maximin Latin hypercube, S-optimal \cite{lauter} Latin hypercube, and MaxPro \cite{maxpro}.
Comparison was performed by generating 500 draws from the corresponding Gaussian process with mean zero, then for each generating a random uniform design, a random Latin hypercube, a maximin Latin hypercube, an S-optimal Latin hypercube, and a MaxPro design as well as a 100 point random uniform testing set. 
For each Gaussian process draw emulators were built, for both fixed, known correlation parameters and estimated parameters, predictions generated on the testing set, and the \emph{maximum} squared prediction error computed. 
The \texttt{R} packages \texttt{SLHD} \cite{SLHD}, \texttt{MaxPro} \cite{maxpropackage}, \texttt{mlegp} \cite{mlegp}, and \texttt{CGP} \cite{cgppackage} were used for generating Latin hypercube and MaxPro designs, fitting \emph{stationary} Gaussian process emulators with \emph{estimated correlation parameters}, and fitting composite Gaussian process emulators with \emph{estimated correlation parameters}, respectively. All computation was performed in \texttt{R} 3.1.1 (R Core Team, Vienna, Austria).

Results are summarized in Table 1.
Broadly, the proposed high-quality designs, as well as the MaxPro designs, performed well across situations, while random uniform design points,
random Latin hypercube, maximin Latin hypercube, S-optimal Latin hypercube performed relatively poorly.
There are a handful of deviations from and nuances to this basic pattern.
First, designs which are tailor-built for non-stationarity \emph{can} outperform space-filling designs (nominal and numeric/minimax and maximin), but do not always outperform them in practice when nominal, numeric, and potentially parameter estimation sources of inaccuracy come into play simultaneously. In particular, the influence of parameter estimation in the context of the composite Gaussian process is not clear, and could be a fruitful arena for future research.
Second, in a situation where inert variables are expected, designs with very poor projection properties (such as the stationary numeric design) should be avoided.
Poor performance of designs which emphasize low-dimensional projections is almost certainly \emph{partly} due to the fact that no low-order functional ANOVA is present in the Gaussian process in the first four examples.
However, lower dimensional projections of the \emph{distance-based} criteria may be the relevant quantities.
On the other hand, the results in Theorem \ref{discrepancyThm} indicate that the nominal and parameter estimation sources of inaccuracy can be controlled via the \emph{star discrepancy}.

\begin{table}[h]
\caption{Comparisons of \emph{expected maximum} squared error of emulators based on nominal, numeric, parameter estimation, random uniform, random Latin hypercube, maximin Latin hypercube, S-optimal Latin hypercube, and MaxPro designs, potentially along with nominal non-stationary, regression, and numeric non-stationary, for several situations, stationary covariance, linear regression functions, non-stationary covariance, and inert prediction variables.}\label{comparisonTable}
\resizebox{\textwidth}{!}{
\begin{tabular}{ccccccc}
\multicolumn{3}{c}{\bf Stationary Example}&&\multicolumn{3}{c}{\bf Non-Stationary Nominal Example}\\
Design&True Parameters&Estimated Parameters&&Design&True Parameters&Estimated Parameters\\
\cline{1-3}\cline{5-7}
Nominal&2.54$\times10^{-9}$&7.86$\times10^{-6}$&&Nominal&5.02&5.65\\
Numeric&4.48$\times10^{-10}$&2.65$\times10^{-6}$&&Numeric&7.27&6.47\\
Parameter Estimation&5.56$\times10^{-10}$&3.15$\times10^{-6}$&&Parameter Estimation&6.29&5.58\\
Random Uniform&4.67$\times10^{-6}$&1.20$\times10^{-1}$&&Random Uniform&9.13&15.24\\
Random LH&1.81$\times10^{-6}$&7.98$\times10^{-5}$&&Random LH&8.79&11.90\\
Maximin LH&5.06$\times10^{-8}$&2.78$\times10^{-5}$&&Maximin LH&7.13&9.12\\
S-Optimal LH&5.50$\times10^{-8}$&2.89$\times10^{-5}$&&S-Optimal LH&7.69&9.23\\
MaxPro&8.67$\times10^{-8}$&4.58$\times10^{-5}$&&MaxPro&6.37&6.69\\
&&&&Nominal Non-Stationary&4.57&10.04\\[0.2cm]
\multicolumn{3}{c}{\bf Regression Functions Example}&&\multicolumn{3}{c}{\bf Non-Stationary Numeric Example}\\
Design&True Parameters&Estimated Parameters&&Design&True Parameters&Estimated Parameters\\
\cline{1-3}\cline{5-7}
Nominal&5.13$\times10^{-2}$&1.25$\times10^{-2}$&&Nominal&7.13$\times 10^{-1}$&3.52\\
Numeric&8.52$\times10^{-3}$&4.56$\times 10^{-3}$&&Numeric&1.07&3.54\\
Parameter Estimation&1.41$\times10^{-2}$&2.48$\times 10^{-2}$&&Parameter Estimation&1.19&2.17\\
Random Uniform&9.00$\times10^{-1}$&1.19&&Random Uniform&6.80&1.63$\times 10^{1}$\\
Random LH&4.38$\times10^{-1}$&8.33$\times 10^{-2}$&&Random LH&5.05&1.08$\times 10^{1}$\\
Maximin LH&1.75$\times10^{-1}$&3.69$\times 10^{-2}$&&Maximin LH&2.66&7.30\\
S-Optimal LH&2.10$\times10^{-1}$&3.24$\times 10^{-1}$&&S-Optimal LH&3.17&7.93\\
MaxPro&1.20$\times10^{-1}$&2.85$\times 10^{-2}$&&MaxPro&1.38&3.53\\
Regression&1.43$\times10^{-2}$&6.35$\times 10^{-2}$&&Numeric Non-Stationary&9.30$\times 10^{-1}$&1.87\\[0.2cm]
\multicolumn{3}{c}{\bf ${\bf x}_2$ Inert Example}&&\multicolumn{3}{c}{}\\
Design&True Parameters&Estimated Parameters&&& & \\
\cline{1-3}
Nominal&3.34$\times10^{-19}$&2.23$\times10^{-5}$&&&&\\
Numeric&3.78$\times10^{-19}$&7.18&&&&\\
Parameter Estimation&2.56$\times10^{-19}$&4.83$\times 10^{-6}$&&&&\\
Random Uniform&1.10$\times10^{-11}$&3.20&&&&\\
Random LH&4.23$\times10^{-18}$&1.71$\times 10^{-3}$&&&&\\
Maximin LH&1.33$\times10^{-18}$&2.67$\times 10^{-7}$&&&&\\
S-Optimal LH&3.66$\times10^{-18}$&2.37$\times 10^{-5}$&&&&\\
MaxPro&1.81$\times10^{-19}$&2.30$\times 10^{-8}$&&&&
\end{tabular}}
\end{table}

This work has a several limitations. All results are in terms of \emph{controlling} error rates with upper bounds. Actual error rates (of the nominal, numeric, or parameter estimation variety) could be substantially less in a particular situation. 
In a specific practical context, a sequential approach to data collection may be appropriate, see for example \cite{jones} or \cite{gramacy}.
Further, no consideration is given to numeric error in parameter estimation and this error could be substantial, especially if the design is poor with respect to information about the parameters. 
However, given the secondary importance of experimental design properties \emph{specific} to parameter estimation, this source of error is not expected to strongly impact the error in interpolation. Also, the discussed model for non-stationarity 
is capable of approximating 
only non-constant correlation decay across the input space and, in particular, does not allow non-constant underlying variability in the Gaussian process model. However, non-constant underlying variability can be modeled as $\Psi({\bf u},{\bf v})=\sigma({\bf u})\sigma({\bf v})\Phi({\bf u}-{\bf v})$ and this non-stationary model behaves intuitively, with regions having more underlying variability requiring a higher density of points than regions having relatively less variability. The results follow in a manner similar to non-stationarity in correlation, although they are in fact simpler, and this development is omitted due to space constraints. 
Lastly, the impact on interpolator accuracy of a 
number important modeling and design considerations, such as 
low-order functional ANOVAs, orthogonality of inputs \cite{bingham}, or mixed categorical and quantitative input variables \cite{qian} has not been examined.

\appendix


\section{Proof of Proposition \ref{nominalProp}}\label{nominalPropProof}

Express ${\rm MSPE}_2$ from equation (\ref{MSPE}) in terms of partitioned matrices,
\begin{equation}
\begin{split}
& {\rm MSPE}_2 =  \Psi_{\boldsymbol{\theta}}({\bf x},{\bf x})  - \begin{pmatrix} {\bf a}_1' & {\bf a}_2' \end{pmatrix} 
\begin{pmatrix} {\bf B}_{11} & {\bf B}_{12} \\ {\bf B}_{21} & {\bf B}_{22} \end{pmatrix}^{-1} 
\begin{pmatrix} {\bf a}_1 \\ {\bf a}_2 \end{pmatrix}, \label{Partition_matrix}
\end{split}
\end{equation}
where 
\begin{gather}
\begin{split}
&\hspace{0.25 in}{\bf a}_1 = \begin{pmatrix} {\bf h(x)} \\ \Psi_{\boldsymbol{\theta}}({\bf X}_1,{\bf x}) \end{pmatrix},\quad 
{\bf a}_2 = \Psi_{\boldsymbol{\theta}}({\bf X}_2^*,{\bf x}), \quad{\bf B}_{11} = \begin{pmatrix} {\bf 0} &  {\bf H}({\bf X}_1)' \\ {\bf H}({\bf X}_1) & {\Psi_{\boldsymbol{\theta}}({\bf X}_1,{\bf X}_1)} \end{pmatrix},\\
&{\bf B}_{12} = \begin{pmatrix} {\bf H}({\bf X}_2^*)' \\ \Psi_{\boldsymbol{\theta}}({\bf X}_1,{\bf X}_2^*) \end{pmatrix},\quad {\bf B}_{21}={\bf B}_{12}', \quad {\bf B}_{22} = \Psi_{\boldsymbol{\theta}}({\bf X}_2^*,{\bf X}_2^*), \quad\text{and}\quad{\bf X}_2^*={\bf X}_2\setminus{\bf X}_1.\nonumber
\end{split}
\end{gather}
Applying partitioned matrix inverse results \cite{harville} and simplifying (\ref{Partition_matrix}) gives,
\begin{equation*} \nonumber
\begin{split}
{\rm MSPE}_2 &={\rm MSPE}_1 - ({\bf a}_2 - {\bf B}_{21} {\bf B}_{11}^{-1}{\bf a}_1)' \ {\bf B}_{22\cdot1}^{-1} \ ({\bf a}_2 - {\bf B}_{21} {\bf B}_{11}^{-1}{\bf a}_1),
\end{split}
\end{equation*}
where ${\bf B}_{22\cdot1}  = {\bf B}_{22} - {\bf B}_{21}{\bf B}_{11}^{-1}{\bf B}_{12}$.
Now, the proof is completed by showing that ${\bf B}_{22\cdot1}$ is non-negative definite (${\bf B}_{22\cdot1}\succeq 0$).
Once again applying partitioned matrix inverse results gives,
\begin{equation*}
\begin{split}
{\bf B}_{22\cdot1}  &  = \Psi_{\boldsymbol	{\theta}}({\bf X}_2^* , {\bf X}_2^*) - \begin{pmatrix} {{\bf H}({\bf X}_2^*)} & \Psi_{\boldsymbol{\theta}}({{\bf X}_2^* , {\bf X}_1})  \end{pmatrix}
\begin{pmatrix} {\bf 0} & {\bf H}({\bf X}_1)' \\ {\bf H}({\bf X}_1) & \Psi_{\boldsymbol{\theta}}({{\bf X}_1,{\bf X}_1}) \end{pmatrix}^{-1}
\begin{pmatrix} {\bf H}({\bf X}_2^*)' \\  \Psi_{\boldsymbol{\theta}}({{\bf X}_1 , {\bf X}_2^*})\end{pmatrix} \\
&  = \Psi_{\boldsymbol{\theta}}({{\bf X}_2^* , {\bf X}_2^*}) - \Psi_{\boldsymbol{\theta}}({{\bf X}_2^*,{\bf X}_1})\Psi_{\boldsymbol{\theta}}({{\bf X}_1,{\bf X}_1})^{-1}\Psi_{\boldsymbol{\theta}}({{\bf X}_1,{\bf X}_2^*}) \\
& \qquad + \left({\bf H}({\bf X}_2^*) - \Psi_{\boldsymbol{\theta}}({{\bf X}_2^*,{\bf X}_1})\Psi_{\boldsymbol{\theta}}({{\bf X}_1,{\bf X}_1})^{-1}{\bf H}({\bf X}_1) \right)  \left({\bf H}({\bf X}_1)'\Psi_{\boldsymbol{\theta}}({{\bf X}_1,{\bf X}_1})^{-1} 
{\bf H}({\bf X}_1)\right)^{-1} \\
& \hspace{0.5 in} \times \left({\bf H}({\bf X}_2^*) - \Psi_{\boldsymbol{\theta}}({{\bf X}_2^*,{\bf X}_1})\Psi_{\boldsymbol{\theta}}({{\bf X}_1,{\bf X}_1})^{-1}{\bf H}({\bf X}_1) \right)'
\end{split}
\end{equation*}
The first two terms are non-negative definite because they represent a conditional variance. The third term is non-negative definite since $ {\bf H}({\bf X}_1)'\Psi_{\boldsymbol{\theta}}{\bf ({\bf X}_1,{\bf X}_1)}^{-1} {\bf H}({\bf X}_1)\succeq 0$ .

\section{Proof of Theorem \ref{nominalTheorem}}\label{nominalTheoremProof}

For ${\bf x}_i\in{\bf X}$, an arbitrary set $A\subset\Omega$, and positive definite function $\Psi_{\boldsymbol{\theta}}$, the uppermost terms in (\ref{MSPErewritten}) can be \emph{locally} bounded as
\begin{gather}
\begin{split}
&\sup_{{\bf x}\in A}\Psi_{\boldsymbol{\theta}}({\bf x},{\bf x})-\Psi_{\boldsymbol{\theta}}({\bf x},{\bf X})\Psi_{\boldsymbol{\theta}}({\bf X},{\bf X})^{-1}\Psi_{\boldsymbol{\theta}}({\bf X},{\bf x})\\
&=\sup_{{\bf x}\in A} \Psi_{\boldsymbol{\theta}}({\bf x},{\bf x})- \left[\left(\Psi_{\boldsymbol{\theta}}({\bf x},{\bf X})-\Psi_{\boldsymbol{\theta}}({\bf x}_i,{\bf X})\right)\Psi_{\boldsymbol{\theta}}({\bf X},{\bf X})^{-1}\left(\Psi_{\boldsymbol{\theta}}({\bf X},{\bf x})-\Psi_{\boldsymbol{\theta}}({\bf X},{\bf x}_i)\right)\right.\\
&\quad\quad\quad\left.+2\Psi_{\boldsymbol{\theta}}({\bf x}_i,{\bf X})\Psi_{\boldsymbol{\theta}}({\bf X},{\bf X})^{-1}\Psi_{\boldsymbol{\theta}}({\bf X},{\bf x})-\Psi_{\boldsymbol{\theta}}({\bf x}_i,{\bf X})\Psi_{\boldsymbol{\theta}}({\bf X},{\bf X})^{-1}\Psi_{\boldsymbol{\theta}}({\bf X},{\bf x}_i)\right]\\
&=\sup_{{\bf x}\in A} \Psi_{\boldsymbol{\theta}}({\bf x},{\bf x})
-2\Psi_{\boldsymbol{\theta}}({\bf x}_i,{\bf x})
+\Psi_{\boldsymbol{\theta}}({\bf x}_i,{\bf x}_i)\\
&\quad\quad\quad-\left(\Psi_{\boldsymbol{\theta}}({\bf x},{\bf X})-\Psi_{\boldsymbol{\theta}}({\bf x}_i,{\bf X})\right)\Psi_{\boldsymbol{\theta}}({\bf X},{\bf X})^{-1}\left(\Psi_{\boldsymbol{\theta}}({\bf X},{\bf x})-\Psi_{\boldsymbol{\theta}}({\bf X},{\bf x}_i)\right)\\
&\le\sup_{{\bf x}\in A} \Psi_{\boldsymbol{\theta}}({\bf x},{\bf x})
-2\Psi_{\boldsymbol{\theta}}({\bf x}_i,{\bf x})
+\Psi_{\boldsymbol{\theta}}({\bf x}_i,{\bf x}_i)-\frac{\|\Psi_{\boldsymbol{\theta}}({\bf X},{\bf x})-\Psi_{\boldsymbol{\theta}}({\bf X},{\bf x}_i)\|_2^2}{\lambda_{\rm max}(\Psi_{\boldsymbol{\theta}}({\bf X},{\bf X}))}\\
&\le\sup_{{\bf x}\in A} \Psi_{\boldsymbol{\theta}}({\bf x},{\bf x})
-2\Psi_{\boldsymbol{\theta}}({\bf x}_i,{\bf x})
+\Psi_{\boldsymbol{\theta}}({\bf x}_i,{\bf x}_i)-
\frac{\left(\Psi_{\boldsymbol{\theta}}({\bf x}_i,{\bf x})-\Psi_{\boldsymbol{\theta}}({\bf x}_i,{\bf x}_i)\right)^2}{n\sup_{{\bf u},{\bf v}\in\Omega}\Psi_{\boldsymbol{\theta}}({\bf u},{\bf v})},\label{uppermostTermsBound}
\end{split}
\end{gather}
where the first equality follows by cancellation of terms, the second equality follows from the fact that $\Psi_{\boldsymbol{\theta}}({\bf X},{\bf X})^{-1}\Psi_{\boldsymbol{\theta}}({\bf X},{\bf x}_i)$ equals the $i^{\rm th}$ component of the $n$ dimensional identity, the first inequality follows from ${\bf a}'{\bf B}^{-1}{\bf a}\ge\lambda_{\rm min}({\bf B}^{-1})\|{\bf a}\|_2^2$ and $\lambda_{\rm min}({\bf B}^{-1})=1/\lambda_{\rm max}({\bf B})$, and
the second inequality is true since a sum of squares $\|\cdot\|_2^2$ is larger than any one of its elements squared and Gershgorin's theorem \cite{varga} implies 
\begin{gather}
\lambda_{\rm max}(\Psi_{\boldsymbol{\theta}}({\bf X},{\bf X}))\le\max_j\sum_{i=1}^n\Psi_{\boldsymbol{\theta}}({\bf x}_i,{\bf x}_j)\le n\sup_{{\bf u},{\bf v}\in\Omega}\Psi_{\boldsymbol{\theta}}({\bf u},{\bf v}).\label{gershgorin}
\end{gather}
Then for $\Psi_{\boldsymbol{\theta}}({\bf x},{\bf x})=\sigma^2$ across ${\bf x}\in\Omega$,
and $k=n\sup_{{\bf u},{\bf v}\in\Omega}\Psi_{\boldsymbol{\theta}}({\bf u},{\bf v})$, the right-hand side of (\ref{uppermostTermsBound}) can be rewritten as
\begin{gather}
\frac{1}{k}\sup_{{\bf x}\in A}\left(\sigma^2-\Psi_{\boldsymbol{\theta}}({\bf x}_i,{\bf x})\right)\left(2k-\sigma^2+\Psi_{\boldsymbol{\theta}}({\bf x}_i,{\bf x})\right).\label{intermediateNominalQuadratic}
\end{gather}
Expression (\ref{intermediateNominalQuadratic}) is a concave down quadratic in $\Psi_{\boldsymbol{\theta}}({\bf x}_i,{\bf x})$ with axis of symmetry $\sigma^2-k$, which is $\le 0$ for $n\ge 1$. That is, (\ref{uppermostTermsBound}) is bounded above by
\begin{gather}
\frac{1}{k}\left(\sigma^2-\inf_{{\bf x}\in A}\Psi_{\boldsymbol{\theta}}({\bf x}_i,{\bf x})\right)\left(2k-\sigma^2+\inf_{{\bf x}\in A}\Psi_{\boldsymbol{\theta}}({\bf x}_i,{\bf x})\right).\nonumber
\end{gather}
If $\Omega\subseteq\cup_{i=1}^n A_i$ and ${\bf x}_i\in A_i$, then 
\begin{gather}
\begin{split}
&\sup_{{\bf x}\in\Omega}\Psi_{\boldsymbol{\theta}}({\bf x},{\bf x})-\Psi_{\boldsymbol{\theta}}({\bf x},{\bf X})\Psi_{\boldsymbol{\theta}}({\bf X},{\bf X})^{-1}\Psi_{\boldsymbol{\theta}}({\bf X},{\bf x})\\
&\le\frac{1}{k}\left(\sigma^2-\min_i\inf_{{\bf x}\in A_i}\Psi_{\boldsymbol{\theta}}({\bf x}_i,{\bf x})\right)\left(2k-\sigma^2+\min_i\inf_{{\bf x}\in A_i}\Psi_{\boldsymbol{\theta}}({\bf x}_i,{\bf x})\right).\nonumber
\end{split}
\end{gather}

\section{Proof of Theorem \ref{discrepancyThm}}\label{discrepancyThmProof}

If $f\sim{\rm GP}(0,\Psi(\cdot,\cdot))$, then 
$f$ can be represented as $f=\sum_{i=1}^\infty\langle f,\varphi_i\rangle\varphi_i$, where 
$\langle \cdot,\cdot\rangle$ is the inner product in $L_2([0,1]^d)$, 
$\langle f,\varphi_i\rangle\stackrel{\rm ind.}{\sim}\mathcal{N}(0,\lambda_i)$,
and 
$\Psi(\cdot,\cdot)$ has eigenvalue, eigenfunction decomposition $\Psi(\cdot,\cdot)=\sum_{i=1}^\infty\lambda_i\varphi_i(\cdot)\varphi_i(\cdot)$.
Further, $\hat{f}({\bf x})=\Psi({\bf x},{\bf X})\Psi({\bf X},{\bf X})^{-1}f({\bf X})$ minimizes the mean squared prediction error (MSPE) $\mathbb{E}\left(f({\bf x})-\hat{f}({\bf x})\right)^2$ over functions of the data ${\bf X},\;f({\bf X})$.
Consider another predictor based on the data 
\begin{gather}
\begin{split}
&\hat{f}_*(x)=\sum_{i=1}^n\hat{u}_i\varphi_i(x),\quad\hat{u}_i=\frac{1}{n}\sum_{j=1}^n \varphi_i(x_j)f(x_j).\nonumber
\end{split}
\end{gather}
Then, the \emph{integrated} MSPE of $\hat{f}$ is bounded above by the \emph{integrated} MSPE  of $\hat{f}_*$,
\begin{gather}
\begin{split}
\mathbb{E}\|f-\hat{f}_*\|_{L_2([0,1]^d)}^2&=\mathbb{E}\left\|\sum_{i=1}^\infty\langle f,\varphi_i\rangle\varphi_i(\cdot)-\sum_{i=1}^n\hat{u}_i\varphi_i(\cdot)\right\|_{L_2([0,1]^d)}^2\\
&=\sum_{j=n+1}^\infty\lambda_j+\sum_{j=1}^n\mathbb{E}\left(\int_{[0,1]^d} \varphi_i(x)f(x){\rm d}x-\frac{1}{n}\sum_{j=1}^n \varphi_i(x_j)f(x_j)\right)^2.\nonumber
\end{split}
\end{gather}
Respective Koksma-Hlawka and modulus of continuity bounds on the error rate for numeric integration \cite{niederreiter} give the bounds
\begin{gather}
\begin{split}
&\left|\int_{[0,1]^d} \varphi_i(x)f(x){\rm d}x-\frac{1}{n}\sum_{j=1}^n \varphi_i(x_j)f(x_j)\right|\le V(\varphi_i f)D^*(X),\\
&\left|\int_{[0,1]^d} \varphi_i(x)f(x){\rm d}x-\frac{1}{n}\sum_{j=1}^n \varphi_i(x_j)f(x_j)\right|\le 4\omega(\varphi_i f,D^*(X)^{1/d}),
\nonumber
\end{split}
\end{gather}
where $V(\cdot)$ denotes the total variation in the sense of Hardy and Kraus of $f$, $D^*(X)$ denotes the star discrepancy, and $\omega$ denotes the modulus of continuity \cite{niederreiter}.
%

\section{Proof of Proposition \ref{prop31}}\label{prop31proof}

We will make use of the following lemma on the accuracy of floating point matrix inversion which is a combination of Lemmas 2.7.1 and 2.7.2 in
\cite{golub}.
\begin{lemma}\label{floatinglemma}
Suppose ${\bf Ax}={\bf b}$ and $\tilde{\bf A}\tilde{\bf x}=\tilde{\bf b}$ with $\|{\bf A}-\tilde{\bf A}\|_2\le\delta\|{\bf A}\|_2$, $\|{\bf b}-\tilde{\bf b}\|_2\le\delta\|{\bf b}\|_2$, and $\kappa({\bf A})=r/\delta<1/\delta$ for some $\delta>0$. Then, $\tilde{\bf A}$ is non-singular,
\begin{gather}
\begin{split}
\frac{\|\tilde{\bf x}\|_2}{\|{\bf x}\|_2}\le\frac{1+r}{1-r},\quad{\rm and}\quad
\frac{\|{\bf x}-\tilde{\bf x}\|_2}{\|{\bf x}\|_2}\le\frac{2\delta}{1-r}\kappa({\bf A}),\nonumber
\end{split}
\end{gather}
where $\kappa({\bf A})=\|{\bf A}\|_2\|{\bf A}^{-1}\|_2$.
\end{lemma}
Additionally, note that for (conformable) ${\bf a}$, $\tilde{\bf a}$, ${\bf b}$, and $\tilde{\bf b}$,
\begin{gather}
\begin{split}
&\left|{\bf a}'{\bf b}-\tilde{\bf a}'\tilde{\bf b}\right|=\left|{\bf a}'({\bf b}-\tilde{\bf b})-(\tilde{\bf a}-{\bf a})'\tilde{\bf b}\right|\\
&\le\left|{\bf a}'({\bf b}-\tilde{\bf b})\right|+\left|(\tilde{\bf a}-{\bf a})'\tilde{\bf b}\right|\le\|{\bf a}\|_2\|{\bf b}-\tilde{\bf b}\|_2+\|{\bf a}-\tilde{\bf a}\|_2\|\tilde{\bf b}\|_2.\label{lemmaSupplement}
\end{split}
\end{gather}
The inner portion of the numeric error in the second term in (\ref{decomposition2}) can be bounded as follows. Here, and throughout, ${\bf A}^{-1}{\bf b}$ and $\tilde{\bf A}^{-1}\tilde{\bf b}$ denote the solutions to the linear systems ${\bf A}{\bf x}={\bf b}$ and $\tilde{\bf A}\tilde{\bf x}=\tilde{\bf b}$, respectively, as opposed to the actual matrix multiplication.
The \emph{hats} on parameter estimates are suppressed for simplicity. In fact, all the below results hold for an arbitrary, fixed parameter or parameter estimate.
\begin{gather}
\begin{split}
&\left|\hat{f}_{\boldsymbol{\vartheta}}({\bf x})-\tilde{f}_{\boldsymbol{\vartheta}}({\bf x})\right|\\
&= \left| \left({\bf h}({\bf x})-\tilde{\bf h}({\bf x})\right)'\boldsymbol{\beta} - \left( \Psi_{\boldsymbol{\theta}}({\bf x},{\bf X})\Psi_{\boldsymbol{\theta}}({\bf X},{\bf X})^{-1}{\bf H}({\bf X})-\tilde{\Psi}_{\boldsymbol{\theta}}({\bf x},{\bf X})\tilde{\Psi}_{\boldsymbol{\theta}}({\bf X},{\bf X})^{-1}\tilde{\bf H}({\bf X}) \right)\boldsymbol{\beta}\right.\\
&\quad\;\left. + \left(\Psi_{\boldsymbol{\theta}}({\bf x},{\bf X})\Psi_{\boldsymbol{\theta}}({\bf X},{\bf X})^{-1}f({\bf X})-\tilde{\Psi}_{\boldsymbol{\theta}}({\bf x},{\bf X})\tilde{\Psi}_{\boldsymbol{\theta}}({\bf X},{\bf X})^{-1}\tilde{f}({\bf X})\right) \right|\\
&\le\|{\bf h}({\bf x})-\tilde{\bf h}({\bf x})\|_2\|\boldsymbol{\beta}\|_2+| \Psi_{\boldsymbol{\theta}}({\bf x,X})\Psi_{\boldsymbol{\theta}}({\bf X,X})^{-1}f({\bf X})-\tilde{\Psi}_{\boldsymbol{\theta}}({\bf x},{\bf X})\tilde{\Psi}_{\boldsymbol{\theta}}({\bf X},{\bf X})^{-1}\tilde{f}({\bf X})|\\
&\quad\;+\|\Psi_{\boldsymbol{\theta}}({\bf x},{\bf X})\Psi_{\boldsymbol{\theta}}({\bf X},{\bf X})^{-1}{\bf H}({\bf X})-\tilde{\Psi}_{\boldsymbol{\theta}}({\bf x},{\bf X})\tilde{\Psi}_{\boldsymbol{\theta}}({\bf X},{\bf X})^{-1}\tilde{\bf H}({\bf X})\|_2\|\boldsymbol{\beta}\|_2\\
&=\| {\bf h(x)}-\tilde{\bf h}({\bf x})\|_2\|\boldsymbol{\beta}\|_2+|\Psi_{\boldsymbol{\theta}}({\bf x},{\bf X})\Psi_{\boldsymbol{\theta}}({\bf X},{\bf X})^{-1}f({\bf X})-\tilde{\Psi}_{\boldsymbol{\theta}}({\bf x},{\bf X})\tilde{\Psi}_{\boldsymbol{\theta}}({\bf X},{\bf X})^{-1}\tilde{f}({\bf X})|\\
&\quad\;+\sqrt{\sum_{j=1}^p\left(\Psi_{\boldsymbol{\theta}}({\bf x},{\bf X})\Psi_{\boldsymbol{\theta}}({\bf X},{\bf X})^{-1}{h}_j({\bf X})-\tilde{\Psi}_{\boldsymbol{\theta}}({\bf x},{\bf X})\tilde{\Psi}_{\boldsymbol{\theta}}({\bf X},{\bf X})^{-1}\tilde{h}_j({\bf X}) \right)^2} \ \|\boldsymbol{\beta}\|_2,\label{numInequal1}
\end{split}
\end{gather}
where ${h}_j({\bf X})$ and $\tilde{h}_j({\bf X})$ denote the $j^{\rm th}$ regression function evaluated at ${\bf X}$ and its floating point approximation, respectively.
Let ${\bf u}={\Psi}_{\boldsymbol{\theta}}({\bf X},{\bf X})^{-1}f({\bf X})$, $\tilde{\bf u}=\tilde{\Psi}_{\boldsymbol{\theta}}({\bf X},{\bf X})^{-1}\tilde{f}({\bf X})$, ${\bf v}_j={\Psi}_{\boldsymbol{\theta}}({\bf X},{\bf X})^{-1}{h}_j({\bf X})$, and $\tilde{\bf v}_j=\tilde{\Psi}_{\boldsymbol{\theta}}({\bf X},{\bf X})^{-1}\tilde{h}_j({\bf X})$. Then, (\ref{numInequal1}) along with inequality (\ref{lemmaSupplement}) implies
\begin{gather}
\begin{split}
&\left|\hat{f}_{\boldsymbol{\vartheta}}({\bf x})-\tilde{f}_{\boldsymbol{\vartheta}}({\bf x})\right|\\
&\le\|{\bf h}({\bf x})-\tilde{\bf h}({\bf x})\|_2\|\boldsymbol{\beta}\|_2+
\|\Psi_{\boldsymbol{\theta}}({\bf x},{\bf X})\|_2\|{\bf u}-\tilde{\bf u}\|_2+\|\Psi_{\boldsymbol{\theta}}({\bf x},{\bf X})-\tilde{\Psi}_{\boldsymbol{\theta}}({\bf x},{\bf X})\|_2\|\tilde{\bf u}\|_2\\
&\quad\;+\sqrt{\sum_{j=1}^p\left(\|\Psi_{\boldsymbol{\theta}}({\bf x},{\bf X})\|_2\|{\bf v}_j-\tilde{\bf v}_j\|_2+\|\Psi_{\boldsymbol{\theta}}({\bf x},{\bf X})-\tilde{\Psi}_{\boldsymbol{\theta}}({\bf x},{\bf X})\|_2\|\tilde{\bf v}_j\|_2 \right)^2} \ \|\boldsymbol{\beta}\|_2.\label{numInequal2}
\end{split}
\end{gather}
Under Assumption \ref{numAssumption}, Lemma \ref{floatinglemma} can be applied to (\ref{numInequal2}) to obtain
\begin{gather}
\begin{split}
&\left|\hat{f}_{\boldsymbol{\vartheta}}({\bf x})-\tilde{f}_{\boldsymbol{\vartheta}}({\bf x})\right|\\
&\le\delta\|{\bf h}({\bf x})\|_2\|\boldsymbol{\beta}\|_2+
\|\Psi_{\boldsymbol{\theta}}({\bf x},{\bf X})\|_2\frac{2\delta}{1-r}\kappa(\Psi_{\boldsymbol{\theta}}({\bf X},{\bf X}))\|{\bf u}\|_2+
\delta\|\Psi_{\boldsymbol{\theta}}({\bf x},{\bf X})\|_2\frac{1+r}{1-r}\|{\bf u}\|_2\\
&\quad\;+\sqrt{\sum_{j=1}^p\biggl(\|\Psi_{\boldsymbol{\theta}}({\bf x},{\bf X})\|_2\frac{2\delta}{1-r}\kappa(\Psi_{\boldsymbol{\theta}}({\bf X},{\bf X}))\|{\bf v}_j\|_2+\delta\|\Psi_{\boldsymbol{\theta}}({\bf x},{\bf X})\|_2\frac{1+r}{1-r}\|{\bf v}_j\|_2 \biggr)^2} \ \|\boldsymbol{\beta}\|_2.\nonumber
\end{split}
\end{gather}
Note that $\|{\bf u}\|_2=\|\Psi_{\boldsymbol{\theta}}({\bf X},{\bf X})^{-1}f({\bf X})\|_2\le\|\Psi_{\boldsymbol{\theta}}({\bf X},{\bf X})^{-1}\|_2\|f({\bf X})\|_2=\|f({\bf X})\|_2/\lambda_{\rm min}(\Psi_{\boldsymbol{\theta}}({\bf X},{\bf X}))$. Similarly, $\|{\bf v}_j\|_2\le\|h_j({\bf X})\|_2/\lambda_{\rm min}(\Psi_{\boldsymbol{\theta}}({\bf X},{\bf X}))$.
Using these facts along with $r<1$ and grouping terms gives the proposition.

\section{Proof of Theorem \ref{numericTheorem}}\label{numericTheoremProof}

For a continuous, positive definite, translation invariant kernel $\Phi$ which has Fourier transform $\hat{\Phi}\in L_1(\mathbb{R}^d)$,
\begin{gather}
\begin{split}
&\sum_{j=1}^n\sum_{k=1}^n\alpha_j\alpha_k\Phi({\bf x}_j-{\bf x}_k)=(2\pi)^{-d/2}\sum_{j=1}^n\sum_{k=1}^n\alpha_j\alpha_k \int_{\mathbb{R}^d} e^{i\boldsymbol{\omega}'({\bf x}_j-{\bf x}_k)}\hat{\Phi}(\boldsymbol{\omega}){\rm d}\boldsymbol{\omega}\\
&=(2\pi)^{-d/2}\int_{\mathbb{R}^d} \left|\sum_{j=1}^n\alpha_j e^{i\boldsymbol{\omega}'{\bf x}_j}\right|^2\hat{\Phi}(\boldsymbol{\omega}){\rm d}\boldsymbol{\omega},\label{quadraticRepresentation}
\end{split}
\end{gather}
for $\boldsymbol{\alpha}\in\mathbb{R}^n$, ${\bf x}_i\in\mathbb{R}^d$. Representation (\ref{quadraticRepresentation}) implies that a lower bound for $\sum_{j,k}\alpha_j\alpha_k\Phi({\bf x}_j-{\bf x}_k)$ is provided by  $\sum_{j,k}\alpha_j\alpha_k\Upsilon({\bf x}_j-{\bf x}_k)$, where $\Upsilon$ has $\hat{\Upsilon}(\boldsymbol{\omega})\le\hat{\Phi}(\boldsymbol{\omega})$. 
Consider $\Upsilon_M$ with 
\begin{gather}
\hat{\Upsilon}_M(\boldsymbol{\omega})=\frac{\hat{\Phi}_*(M)\Gamma(d/2+1)}{2^d M^d \pi^{d/2}}(\chi_M*\chi_M)(\boldsymbol{\omega}),\nonumber
\end{gather}
where $M>0$, $\hat{\Phi}_*(M)=\inf_{\|\boldsymbol{\omega}\|_2\le 2M}\hat{\Phi}(\boldsymbol{\omega})$,  $\chi_M(\boldsymbol{\omega})=1$ for $\|\boldsymbol{\omega}\|_2\le M$ and $0$ otherwise, and $*$ denotes the convolution operator 
\begin{gather}
(f*g)({\bf x})=\int_{\mathbb{R}^d}f({\bf y})g({\bf x}-{\bf y}){\rm d}{\bf y}.\nonumber
\end{gather}
For $\|\boldsymbol{\omega}\|_2>2M$, $\hat{\Upsilon}_M(\boldsymbol{\omega})=0\le\hat{\Phi}(\boldsymbol{\omega})$. On the other hand, for $\|\boldsymbol{\omega}\|_2\le 2M$,
\begin{gather}
\begin{split}
&\hat{\Upsilon}_M(\boldsymbol{\omega})=\frac{\hat{\Phi}_*(M)\Gamma(d/2+1)}{2^d M^d \pi^{d/2}}\int_{\mathbb{R}^d} \chi_M({\bf t})\chi_M(\boldsymbol{\omega}-{\bf t}){\rm d}{\bf t}\\
&\le\frac{\hat{\Phi}_*(M)\Gamma(d/2+1)}{2^d M^d \pi^{d/2}}{\rm vol}\;B({\bf 0},2M)=\hat{\Phi}_*(M)\le\hat{\Phi}(\boldsymbol{\omega}),\nonumber
\end{split}
\end{gather}
where $B({\bf 0},2M)=\{\|\boldsymbol{\omega}-{\bf 0}\|_2\le 2M\}$ denotes a ball of radius $2M$ centered at the origin.
So,  $\hat{\Upsilon}(\boldsymbol{\omega})\le\hat{\Phi}(\boldsymbol{\omega})$ for all $\boldsymbol{\omega}\in\mathbb{R}^d$.
The candidate $\Upsilon$ can be recovered from the inverse Fourier transform $(\hat{\Upsilon}_M)^\vee$
\begin{gather}
\begin{split}
&\Upsilon_M({\bf t})=\frac{\hat{\Phi}_*(M)\Gamma(d/2+1)}{2^d M^d \pi^{d/2}}(\chi_M*\chi_M)^\vee({\bf t})\\
&=\frac{\hat{\Phi}_*(M)\Gamma(d/2+1)}{2^d M^d \pi^{d/2}}(2\pi)^{d/2}((\chi_M)^\vee({\bf t}))^2\\
&=\frac{\hat{\Phi}_*(M)\Gamma(d/2+1)}{2^{d/2}}\|{\bf t}\|^{-d}_2 J^2_{d/2}(M\|{\bf t}\|_2),\nonumber
\end{split}
\end{gather}
where $J_\nu$ is a Bessel function of the first kind. A proof of the final equality is given as Lemma 12.2 in \cite{wendland2005}.
Define $\Upsilon_M({\bf 0})$ as 
\begin{gather}
\Upsilon_M({\bf 0})\equiv\lim_{{\bf t}\to{\bf 0}}\Upsilon_M({\bf t})=\frac{\hat{\Phi}_*(M)}{\Gamma(d/2+1)}\left(\frac{M}{2^{3/2}}\right)^d.\nonumber
\end{gather}
This limit follows from the Taylor series representation of the Bessel function \cite{watson}.

Now, a lower bound on the quadratic form involving $\Upsilon$ is developed.
\begin{gather}
\begin{split}
&\sum_{j=1}^n\sum_{k=1}^n\alpha_j\alpha_k\Upsilon_M({\bf x}_j-{\bf x}_k)=\sum_{j=1}^n\alpha_j^2\Upsilon_M({\bf 0})+\sum_{j\ne k}\alpha_j\alpha_k\Upsilon_M({\bf x}_j-{\bf x}_k)\\
&\ge\sum_{j=1}^n\alpha_j^2\Upsilon_M({\bf 0})-\sum_{j\ne k}|\alpha_j||\alpha_k|\Upsilon_M({\bf x}_j-{\bf x}_k)\\
&\ge\sum_{j=1}^n\alpha_j^2\Upsilon_M({\bf 0})-\frac{1}{2}\sum_{j\ne k}(\alpha_j^2+\alpha_k^2)\Upsilon_M({\bf x}_j-{\bf x}_k)\\
&=\sum_{j=1}^n\alpha_j^2\Upsilon_M({\bf 0})-\sum_{j=1}^n\alpha_j^2\sum_{k=1,\;k\ne j}^n\Upsilon_M({\bf x}_j-{\bf x}_k)\\
&=\sum_{j=1}^n\alpha_j^2\left(\Upsilon_M({\bf 0})-\sum_{k=1,\;k\ne j}^n\Upsilon_M({\bf x}_j-{\bf x}_k)\right).\label{quadraticBound1}
\end{split}
\end{gather}
Each $\sum_{k=1,\;k\ne j}^n|\Upsilon({\bf x}_j-{\bf x}_k)|$ can be bounded in terms of the separation distances
\begin{gather}
q_j=\frac{1}{2}\min_{k=1,\ldots,n,\;k\ne j}\|{\bf x}_j-{\bf x}_k\|_2\quad{\rm and}\quad q=\min_j q_j.\nonumber
\end{gather}
For $m\in\mathbb{N}$, let 
\begin{gather}
E_{jm}=\{{\bf x}\in\mathbb{R}^d:mq_j\le\|{\bf x}_j-{\bf x}\|_2<(m+1)q_j\}.\nonumber
\end{gather}
Then, every ${\bf x}_k,\;k\ne j$ is contained in exactly one $E_{jm}$. Further, every $B({\bf x}_k,q)$ is essentially disjoint and completely contained in 
\begin{gather}
\{{\bf x}\in\mathbb{R}^d:mq_j-q\le\|{\bf x}_j-{\bf x}\|_2<(m+1)q_j+q\}.\nonumber
\end{gather}
So, each $E_{jm}$ can contain no more than
\begin{gather}
\frac{((m+1)q_j+q)^d-(mq_j-q)^d}{q^d}=\left(\frac{(m+1)q_j}{q}+1\right)^d-\left(\frac{mq_j}{q}-1\right)^d\nonumber
\end{gather}
data points.
We now make use of the following lemma.

\begin{lemma}\label{countingPtsLemma}
For $d\in\mathbb{N}$ and $q_j\geq q>0$,
\begin{gather}
\left(\frac{(m+1)q_j}{q}+1\right)^d-\left(\frac{mq_j}{q}-1\right)^d\le(3q_j/q)^d m^{d-1}.\nonumber
\end{gather}
\end{lemma}
\begin{proof}
Take $d=1$, then 
\begin{gather}
\left(\frac{(m+1)q_j}{q}+1\right)-\left(\frac{mq_j}{q}-1\right)=2+q_j/q\le 3q_j/q.\nonumber
\end{gather}
Now, assume the result is true for $1\le d_*<d$. Let $c=3q_j/q$. Then,
\begin{gather}
\begin{split}
&\left(\frac{(m+1)q_j}{q}+1\right)^{d-1}-\left(\frac{mq_j}{q}-1\right)^{d-1}\le c^{d-1} m^{d-2}\\
\Longrightarrow &\left(\frac{(m+1)q_j}{q}+1\right)^d-\left(\frac{(m+1)q_j}{q}+1\right)\left(\frac{mq_j}{q}-1\right)^{d-1}\le \left(\frac{(m+1)q_j}{q}+1\right)c^{d-1} m^{d-2}\\
\Longrightarrow &\left(\frac{(m+1)q_j}{q}+1\right)^d-\left(\frac{mq_j}{q}-1\right)^d\le\left(\frac{(m+1)q_j}{q}+1\right)c^{d-1} m^{d-2}+\left(\frac{q_j}{q}+2\right)\left(\frac{mq_j}{q}-1\right)^{d-1}.\nonumber
\end{split}
\end{gather}
The proof is completed by showing that the right-hand side of the final inequality is bounded above by $c^dm^{d-1}$. The right-hand side can be represented in terms of $c$ as
\begin{gather}
\begin{split}
&\left(\frac{(m+1)c}{3}+1\right)c^{d-1} m^{d-2}+\left(\frac{c}{3}+2\right)\left(\frac{mc}{3}-1\right)^{d-1}\\
&=c^dm^{d-1}\left[\frac{(m+1)}{3m}+\frac{1}{mc}+\left(\frac{1}{3}+\frac{2}{c}\right)\left(\frac{1}{3}-\frac{1}{mc}\right)^{d-1}\right]\\
&\le c^dm^{d-1}\left[\frac{1}{3}+\frac{1}{3m}+\frac{1}{mc}+\left(\frac{1}{3}-\frac{1}{mc}\right)^{d-1}\right]\\
&\le c^dm^{d-1}\left[\frac{1}{3}+\frac{1}{3m}+\frac{1}{mc}+\frac{1}{3}-\frac{1}{mc}\right]\\
&= c^dm^{d-1}\left[\frac{2}{3}+\frac{1}{3m}\right]\le c^dm^{d-1},\nonumber
\end{split}
\end{gather}
where the first inequality is true because $1/3+2/c\le 1$ and the second inequality is true because $(1/3-1/(mc))^{d-1}$ is a decreasing function of $d\ge 2$.
\end{proof}

Lemma \ref{countingPtsLemma} implies that each $E_{jm}$ contains no more than $(3q_j/q)^d m^{d-1}$ points.
Note that on $E_{jm}$, $\Upsilon({\bf x}_j-{\bf x}_k)$ is bounded above as
\begin{gather}
\begin{split}
&\Upsilon({\bf x}_j-{\bf x}_k)=\frac{\hat{\Phi}_*(M)\Gamma(d/2+1)}{2^{d/2}}\|{\bf x}_j-{\bf x}_k\|^{-d}_2 J^2_{d/2}(M\|{\bf x}_j-{\bf x}_k\|_2)\\
&\le\frac{\hat{\Phi}_*(M)\Gamma(d/2+1)}{2^{d/2}}\|{\bf x}_j-{\bf x}_k\|^{-d}_2 \frac{2^{d+2}}{M\pi\|{\bf x}_j-{\bf x}_k\|}\\
&=\Upsilon_M({\bf 0})\frac{\Gamma^2(d/2+1)}{\pi}\left(\frac{4}{M\|{\bf x}_j-{\bf x}_k\|_2}\right)^{d+1}\\
&\le\Upsilon_M({\bf 0})\frac{\Gamma^2(d/2+1)}{\pi}\left(\frac{4}{Mmq_j}\right)^{d+1},\label{boundOnEjm}
\end{split}
\end{gather}
where the first inequality follows from the Bessel function bound provided in Lemma 3.3 of \cite{narcowich}.
Combining Lemma \ref{countingPtsLemma} with (\ref{boundOnEjm}) gives
\begin{gather}
\begin{split}
&\sum_{k=1,\;k\ne j}^n\Upsilon({\bf x}_j-{\bf x}_k)\le\sum_{m=1}^\infty \Upsilon_M({\bf 0})\frac{\Gamma^2(d/2+1)}{\pi}\left(\frac{4}{Mmq_j}\right)^{d+1}(3q_j/q)^dm^{d-1}\\
&=\Upsilon_M({\bf 0})\frac{\Gamma^2(d/2+1)\pi}{18}\left(\frac{q}{q_j}\right)\left(\frac{12}{Mq}\right)^{d+1},\nonumber
\end{split}
\end{gather}
where the equality follows from the fact that $\sum_{m=1}^{\infty}m^{-2}=\pi^2/6$.
Now, taking $M=c_*/q$ and referring back to (\ref{quadraticBound1}), the quadratic form can be bounded as
\begin{gather}
\sum_{j=1}^n\sum_{k=1}^n\alpha_j\alpha_k\Phi({\bf x}_j-{\bf x}_k)\ge\Upsilon_{c_*/q}({\bf 0})\sum_{j=1}^n\alpha_j^2\left(1-\frac{\Gamma^2(d/2+1)\pi}{18}\left(\frac{q}{q_j}\right)\left(\frac{12}{c_*}\right)^{d+1}\right).\nonumber
\end{gather}
%
The stated version of the theorem
follows by applying the previous development to the transformed space, ${\bf v}\mapsto{\bf v}^*=\boldsymbol{\Theta}{\bf v}$. 

\section{Proof of Theorem \ref{paramThm}}\label{paramThmProof}
First, we develop expressions for the components of the approximate mean squared prediction error in Lemma \ref{paramProp} below.
\begin{lemma}\label{paramProp}
If $f\sim {\rm GP} \left( {\bf h}(\cdot)'\boldsymbol{\beta},\Psi_{\boldsymbol{\theta}}(\cdot,\cdot) \right)$, for fixed, known regression functions ${\bf h}(\cdot)$, and $\Psi_{\boldsymbol{\theta}}(\cdot,\cdot)$ as defined in (\ref{paramEstAssumption}), then
\begin{gather}
\begin{split}
&\frac{\partial\hat{f}_{\boldsymbol{\vartheta}}({\bf x})}{\partial\boldsymbol{\vartheta}'}\mathcal{I}({\boldsymbol{\vartheta}})^{-1} \frac{\partial\hat{f}_{\boldsymbol{\vartheta}}({\bf x})}{\partial\boldsymbol{\vartheta}}= 
{\bf c}'_1\mathcal{I}^{-1}_{11}{\bf c}_1+{\bf c}'_3(\mathcal{I}_{33}-\mathcal{I}_{32}\mathcal{I}_{22}^{-1}\mathcal{I}_{23})^{-1}{\bf c}_3,\label{paramErrExpr1}
\end{split}
\end{gather}
for ${\bf x}\in\mathbb{R}^d$ where
\begin{gather}
\begin{split}
&{\bf c}_1 = {\bf h(x)} - {\bf H(X)}' \Psi_{\boldsymbol{\theta}}({\bf X,X})^{-1} \Psi_{\boldsymbol{\theta}}({\bf X,x}),\\
&{\bf c}_3=\left(\frac{\partial\Psi_{\boldsymbol{\theta}}({\bf x,X})}{\partial\boldsymbol{\varrho}} - ({\bf I}_d\otimes\Psi_{\boldsymbol{\theta}}({\bf x,X})\Psi_{\boldsymbol{\theta}}({\bf X,X})^{-1})\frac{\partial\Psi_{\boldsymbol{\theta}}({\bf X,X})}{\partial\boldsymbol{\varrho}}\right)\Psi_{\boldsymbol{\theta}}({\bf X,X})^{-1}\delta({\bf X}),\\
&\mathcal{I}_{11} ={\bf H}({\bf X})'\Psi_{\boldsymbol{\theta}}({\bf X,X})^{-1}{\bf H}({\bf X}),\\
&\mathcal{I}_{22}=\frac{n}{2\sigma^4},\\
&\mathcal{I}_{32}=\frac{1}{2\sigma^4}{\bf C}'_{\boldsymbol{\theta}}{\rm vec}\;\Phi_{\boldsymbol{\varrho}}({\bf X},{\bf X})^{-1},\\
&\mathcal{I}_{33}=\frac{1}{2\sigma^4}{\bf C}'_{\boldsymbol{\theta}} (\Phi_{\boldsymbol{\varrho}}({\bf X},{\bf X})^{-1}\otimes \Phi_{\boldsymbol{\varrho}}({\bf X},{\bf X})^{-1}){\bf C}_{\boldsymbol{\theta}},\\
&{\bf C}_{\boldsymbol{\theta}}=\frac{\partial({\rm vec}\;\Psi_{\boldsymbol{\theta}}({\bf X,X}))}{\partial\boldsymbol{\varrho}'},\nonumber
\end{split}
\end{gather}
where $\delta({\bf X})=f({\bf X})-{\bf H}({\bf X})\boldsymbol{\beta}$.
\end{lemma}
\begin{proof}
Up to an additive constant, the log-likelihood is
\begin{gather}
\ell=
-\frac{1}{2}{\rm log}\;{\rm det}\;\Psi_{\boldsymbol{\theta}}({\bf X,X})-\frac{1}{2}(f({\bf X})-{\bf H}({\bf X})\boldsymbol{\beta})'\Psi_{\boldsymbol{\theta}}({\bf X,X})^{-1}(f({\bf X})-{\bf H}({\bf X})\boldsymbol{\beta}).\nonumber
\end{gather} 
Throughout 
Appendix \ref{paramThmProof}, we will use matrix differentiation, see for example \cite{magnus}.
Then, the vector of derivatives of the emulator with respect to the unknown parameter values $\frac{\partial\hat{f}_{\boldsymbol{\vartheta}}({\bf x})}{\partial\boldsymbol{\vartheta}}$
has block components
\begin{gather}
\begin{split}
{\bf c}_1 &= \frac{\partial\hat{f}_{\boldsymbol{\vartheta}}({\bf x})}{\partial{\boldsymbol{\beta}}} 
= \frac{\partial}{\partial\boldsymbol{\beta}} \left\{ {\bf h(x)}'{\boldsymbol{\beta}} + \Psi_{\boldsymbol{\theta}}({\bf x,X})
\Psi_{\boldsymbol{\theta}}({\bf X,X})^{-1} \left( f({\bf X})-{\bf H}({\bf X}){\boldsymbol{\beta}} \right) \right\}\\
& \hspace{0.7 in} = {\bf h(x)} - {\bf H(X)}' \Psi_{\boldsymbol{\theta}}({\bf X,X})^{-1} \Psi_{\boldsymbol{\theta}}({\bf X,x}),\\
{\bf c}_{2}&=\frac{\partial\hat{f}_{\boldsymbol{\vartheta}}({\bf x})}{\partial\sigma^2}=0.\nonumber
\end{split}
\end{gather}
Developing an expression for $\frac{\partial\hat{f}_{\boldsymbol{\vartheta}}({\bf x})}{\partial\boldsymbol{\varrho}}$ is more complex and broken into a few parts. Let $\delta({\bf X})=f({\bf X})-{\bf H(X)}{\boldsymbol{\beta}}$.
Then,
\begin{gather}
\begin{split}
{\bf c}_3 = \frac{\partial\hat{f}_{\boldsymbol{\vartheta}}({\bf x})}{\partial\boldsymbol{\varrho}} = \left(\frac{\partial\Psi_{\boldsymbol{\theta}}({\bf x,X})}{\partial\boldsymbol{\varrho}}\Psi_{\boldsymbol{\theta}}({\bf X,X})^{-1}+\left( {\bf I}_d\otimes\Psi_{\boldsymbol{\theta}}({\bf x,X}) \right)\frac{\partial\Psi_{\boldsymbol{\theta}}({\bf X,X})^{-1}}{\partial\boldsymbol{\varrho}}\right)\delta({\bf X}).\label{dEmulatordTheta}
\end{split}
\end{gather}
Note that,
\begin{gather}
\begin{split}
{\bf 0}&=\frac{\partial\Psi_{\boldsymbol{\theta}}({\bf X,X})\Psi_{\boldsymbol{\theta}}({\bf X,X})^{-1}}{\partial\boldsymbol{\varrho}}\\
&=\frac{\partial\Psi_{\boldsymbol{\theta}}({\bf X,X})}{\partial\boldsymbol{\varrho}}\Psi_{\boldsymbol{\theta}}({\bf X,X})^{-1}+({\bf I}_d\otimes\Psi_{\boldsymbol{\theta}}({\bf X,X}))\frac{\partial\Psi_{\boldsymbol{\theta}}({\bf X,X})^{-1}}{\partial\boldsymbol{\varrho}}.\nonumber
\end{split}
\end{gather}
So,
\begin{gather}
\begin{split}
\frac{\partial\Psi_{\boldsymbol{\theta}}({\bf X,X})^{-1}}{\partial\boldsymbol{\varrho}}=-({\bf I}_d\otimes\Psi_{\boldsymbol{\theta}}({\bf X,X})^{-1})\frac{\partial\Psi_{\boldsymbol{\theta}}({\bf X,X})}{\partial\boldsymbol{\varrho}}\Psi_{\boldsymbol{\theta}}({\bf X,X})^{-1}\label{dInv}
\end{split}
\end{gather}
Plugging (\ref{dInv}) into equation (\ref{dEmulatordTheta}) gives the third block component
\begin{gather}
\begin{split}
{\bf c}_3&=\frac{\partial\hat{f}_{\boldsymbol{\vartheta}}({\bf x})}{\partial\boldsymbol{\varrho}}\\
&=\left(\frac{\partial\Psi_{\boldsymbol{\theta}}({\bf x,X})}{\partial\boldsymbol{\varrho}} - ({\bf I}_d\otimes\Psi_{\boldsymbol{\theta}}({\bf x,X})\Psi_{\boldsymbol{\theta}}({\bf X,X})^{-1})\frac{\partial\Psi_{\boldsymbol{\theta}}({\bf X,X})}{\partial\boldsymbol{\varrho}}\right)\Psi_{\boldsymbol{\theta}}({\bf X,X})^{-1}\delta({\bf X}).\nonumber
\end{split}
\end{gather}
Now, we develop 
an expression for $\mathcal{I}({\boldsymbol{\vartheta}}_*)$.
First,
\begin{gather}
\begin{split}
\frac{\partial\ell}{\partial\boldsymbol{\beta}} &= 
{\bf H}({\bf X})'\Psi_{\boldsymbol{\theta}}({\bf X,X})^{-1}(f({\bf X})-{\bf H}({\bf X})\boldsymbol{\beta}),\\
\frac{\partial\ell}{\partial\sigma^2} &= -\frac{n}{2\sigma^2}+\frac{1}{2\sigma^4}\left(f({\bf X})-{\bf H}({\bf X})\boldsymbol{\beta}\right)'\Phi_{\boldsymbol{\varrho}}({\bf X},{\bf X})^{-1}(f({\bf X})-{\bf H}({\bf X})\boldsymbol{\beta}).\nonumber
\end{split}
\end{gather}
The derivative of $\ell$ with respect to $\boldsymbol{\varrho}$ can be broken into three parts via the chain rule,
\begin{gather}
\frac{\partial\ell}{\partial\boldsymbol{\varrho}}=\frac{\partial({\rm vec}\;\Psi_{\boldsymbol{\theta}}({\bf X,X}))'}{\partial\boldsymbol{\varrho}} \underbrace{\frac{\partial({\rm vec}\;\Psi_{\boldsymbol{\theta}}({\bf X,X})^{-1})'}{\partial{\rm vec}\;\Psi_{\boldsymbol{\theta}}({\bf X,X})}}_{A} \underbrace{\frac{\partial\ell}{\partial{\rm vec}\;\Psi_{\boldsymbol{\theta}}({\bf X,X})^{-1}}}_{B}.\label{allthreeparts}
\end{gather}
Let
\begin{gather}
{\bf C}_{\boldsymbol{\theta}}=\frac{\partial({\rm vec}\;\Psi_{\boldsymbol{\theta}}({\bf X,X}))}{\partial\boldsymbol{\varrho}'}.\label{part1Vector}
\end{gather}
Parts $A$ and $B$ can be treated in turn. 
Consider part $A$. Similarly to (\ref{dInv}),
\begin{gather}
\begin{split}
{\bf 0}&=\frac{\partial({\rm vec}\;{\bf I}_n)'}{\partial{\rm vec}\;\Psi_{\boldsymbol{\theta}}({\bf X,X})}=\frac{\partial({\rm vec}\;\Psi_{\boldsymbol{\theta}}({\bf X,X})^{-1}\Psi_{\boldsymbol{\theta}}({\bf X,X}))'}{\partial{\rm vec}\;\Psi_{\boldsymbol{\theta}}({\bf X,X})}\\
&=\frac{\partial({\rm vec}\;\Psi_{\boldsymbol{\theta}}({\bf X,X})^{-1})'}{\partial{\rm vec}\;\Psi_{\boldsymbol{\theta}}({\bf X,X})}(\Psi_{\boldsymbol{\theta}}({\bf X,X})\otimes {\bf I}_n)\\
&\quad + \frac{\partial({\rm vec}\;\Psi_{\boldsymbol{\theta}}({\bf X,X}))'}{\partial{\rm vec}\;\Psi_{\boldsymbol{\theta}}({\bf X,X})}({\bf I}_n\otimes \Psi_{\boldsymbol{\theta}}({\bf X,X})^{-1})\\
\Longrightarrow &\frac{\partial({\rm vec}\;\Psi_{\boldsymbol{\theta}}({\bf X,X})^{-1})'}{\partial{\rm vec}\;\Psi_{\boldsymbol{\theta}}({\bf X,X})}=-(\Psi_{\boldsymbol{\theta}}({\bf X,X})^{-1}\otimes \Psi_{\boldsymbol{\theta}}({\bf X,X})^{-1}).\label{part2}
\end{split}
\end{gather}
Next, consider part $B$,
\begin{gather}
\frac{\partial\ell}{\partial{\rm vec}\;\Psi_{\boldsymbol{\theta}}({\bf X,X})^{-1}}=\frac{1}{2}\left[{\rm vec}\;\Psi_{\boldsymbol{\theta}}({\bf X,X})-(f({\bf X})-{\bf H}({\bf X})\boldsymbol{\beta}) \otimes (f({\bf X})-{\bf H}({\bf X})\boldsymbol{\beta})\right].\label{part3}
\end{gather}\\
Equations (\ref{part1Vector}), (\ref{part2}), and (\ref{part3}), and 
can be plugged into equation (\ref{allthreeparts}) to give
\begin{gather}
\begin{split}
\frac{\partial\ell}{\partial\boldsymbol{\varrho}}&=-\frac{1}{2}{\bf C}'_{\boldsymbol{\theta}} (\Psi_{\boldsymbol{\theta}}({\bf X,X})^{-1}\otimes \Psi_{\boldsymbol{\theta}}({\bf X,X})^{-1})\left[{\rm vec}\;\Psi_{\boldsymbol{\theta}}({\bf X,X})-(f({\bf X})-{\bf H}({\bf X})\boldsymbol{\beta})\otimes (f({\bf X})-{\bf H}({\bf X})\boldsymbol{\beta})\right]\\
&=-\frac{1}{2}{\bf C}'_{\boldsymbol{\theta}}\left[{\rm vec}\;\Psi_{\boldsymbol{\theta}}({\bf X,X})^{-1}-\Psi_{\boldsymbol{\theta}}({\bf X,X})^{-1}\delta({\bf X})\otimes \Psi_{\boldsymbol{\theta}}({\bf X,X})^{-1}\delta({\bf X})\right].\nonumber
\end{split}
\end{gather}
So, the 
information matrix has block components
\begin{gather}
\begin{split}
&\mathcal{I}_{11} = \mathcal{I}(\boldsymbol{\beta},\boldsymbol{\beta})=
-\mathbb{E}\frac{\partial^2\ell}{\partial\boldsymbol{\beta}\partial\boldsymbol{\beta}'}={\bf H}({\bf X})'\Psi_{\boldsymbol{\theta}}({\bf X,X})^{-1}{\bf H}({\bf X}),\\
&\mathcal{I}_{21}=\mathcal{I}(\sigma^2,\boldsymbol{\beta})=
-\mathbb{E}\frac{\partial^2\ell}{\partial\sigma^2\partial\boldsymbol{\beta}'}
=\frac{1}{\sigma^4}\mathbb{E}(f({\bf X})-{\bf H}({\bf X})\boldsymbol{\beta})'\Phi_{\boldsymbol{\varrho}}({\bf X},{\bf X})^{-1}{\bf H}({\bf X})={\bf 0}',\\
&\mathcal{I}_{31}=\mathcal{I}(\boldsymbol{\varrho},\boldsymbol{\beta})=
-\mathbb{E}\frac{\partial^2\ell}{\partial\boldsymbol{\varrho}\partial\boldsymbol{\beta}'}\\
&={\bf C}'_{\boldsymbol{\theta}}\left(\Psi_{\boldsymbol{\theta}}({\bf X,X})^{-1}{\bf H}({\bf X}) \otimes\Psi_{\boldsymbol{\theta}}({\bf X,X})^{-1}\mathbb{E}(f({\bf X})-{\bf H}({\bf X})\boldsymbol{\beta}) \right)={\bf 0},\\
&\mathcal{I}_{22}=\mathcal{I}(\sigma^2,\sigma^2)=
-\mathbb{E}\frac{\partial^2\ell}{\partial\sigma^2\partial\sigma^2}\\
&=-\frac{n}{2\sigma^4}+\frac{1}{\sigma^6}\mathbb{E}(f({\bf X})-{\bf H}({\bf X})\boldsymbol{\beta})'\Phi_{\boldsymbol{\varrho}}({\bf X},{\bf X})^{-1}(f({\bf X})-{\bf H}({\bf X})\boldsymbol{\beta})\\
&=-\frac{n}{2\sigma^4}+\frac{1}{\sigma^6}{\rm trace}\;\Phi_{\boldsymbol{\varrho}}({\bf X},{\bf X})^{-1}\sigma^2\Phi_{\boldsymbol{\varrho}}({\bf X},{\bf X})=\frac{n}{2\sigma^4},\\
&\mathcal{I}_{32}=\mathcal{I}(\boldsymbol{\varrho},\sigma^2)=
-\mathbb{E}\frac{\partial^2\ell}{\partial\boldsymbol{\varrho}\partial\sigma^2}\\
&=\frac{1}{2}{\bf C}'_{\boldsymbol{\theta}}
\mathbb{E}\left(-\frac{1}{\sigma^4}{\rm vec}\;\Phi_{\boldsymbol{\varrho}}({\bf X},{\bf X})^{-1}+\frac{2}{\sigma^6}\Phi_{\boldsymbol{\varrho}}({\bf X},{\bf X})^{-1}\delta({\bf X})\otimes \Phi_{\boldsymbol{\varrho}}({\bf X},{\bf X})^{-1}\delta({\bf X})\right)\\
&=\frac{1}{2}{\bf C}'_{\boldsymbol{\theta}}\left(-\frac{1}{\sigma^4}{\rm vec}\;\Phi_{\boldsymbol{\varrho}}({\bf X},{\bf X})^{-1}+\frac{2}{\sigma^4}{\rm vec}\;\Phi_{\boldsymbol{\varrho}}({\bf X},{\bf X})^{-1}\right)\\
&=\frac{1}{2\sigma^4}{\bf C}'_{\boldsymbol{\theta}}{\rm vec}\;\Phi_{\boldsymbol{\varrho}}({\bf X},{\bf X})^{-1}.\nonumber
\end{split}
\end{gather}
Developing a formula for $\mathcal{I}(\boldsymbol{\varrho},\boldsymbol{\varrho})$ is more complex and broken into parts. 
\begin{gather}
\begin{split}
&\mathcal{I}(\boldsymbol{\varrho},\boldsymbol{\varrho})=
-\mathbb{E}\frac{\partial^2\ell}{\partial\boldsymbol{\varrho}\partial\boldsymbol{\varrho}'}\\
&=\frac{1}{2}\mathbb{E}\left({\bf I}_d\otimes \left[({\rm vec}\;\Psi_{\boldsymbol{\theta}}({\bf X,X})^{-1})'-(\delta({\bf X})'\Psi_{\boldsymbol{\theta}}({\bf X,X})^{-1}\otimes \delta({\bf X})'\Psi_{\boldsymbol{\theta}}({\bf X,X})^{-1})\right]\right)\frac{\partial {\bf C}_{\boldsymbol{\theta}}}{\partial\boldsymbol{\varrho}}\\
&\quad+\frac{1}{2}\mathbb{E}\left(\frac{\partial({\rm vec}\;\Psi_{\boldsymbol{\theta}}({\bf X,X})^{-1})'}{\partial\boldsymbol{\varrho}}-\frac{\partial(\delta({\bf X})'\Psi_{\boldsymbol{\theta}}({\bf X,X})^{-1}\otimes \delta({\bf X})'\Psi_{\boldsymbol{\theta}}({\bf X,X})^{-1})}{\partial\boldsymbol{\varrho}}\right){\bf C}_{\boldsymbol{\theta}}.\label{initialinfoexpression}
\end{split}
\end{gather}
Note that the expectation of the first term in (\ref{initialinfoexpression}) is zero, since
\begin{gather}
\begin{split}
&\mathbb{E}(\delta({\bf X})'\Psi_{\boldsymbol{\theta}}({\bf X,X})^{-1}\otimes \delta({\bf X})'\Psi_{\boldsymbol{\theta}}({\bf X,X})^{-1})\\
&=\mathbb{E}\left({\rm vec}\;(\Psi_{\boldsymbol{\theta}}({\bf X,X})^{-1}\delta({\bf X}) \delta({\bf X})'\Psi_{\boldsymbol{\theta}}({\bf X,X})^{-1})\right)'=({\rm vec}\;\Psi_{\boldsymbol{\theta}}({\bf X,X})^{-1})'.\nonumber
\end{split}
\end{gather}
So,
\begin{gather}
\begin{split}
\mathcal{I}(\boldsymbol{\varrho},\boldsymbol{\varrho})
&=\frac{1}{2}\left(-{\bf C}'_{\boldsymbol{\theta}} (\Psi_{\boldsymbol{\theta}}({\bf X,X})^{-1}\otimes \Psi_{\boldsymbol{\theta}}({\bf X,X})^{-1})\right.\\
&\quad\quad\quad\left.-\mathbb{E}\frac{\partial\left({\rm vec}\;(\Psi_{\boldsymbol{\theta}}({\bf X,X})^{-1}\delta({\bf X}) \delta({\bf X})'\Psi_{\boldsymbol{\theta}}({\bf X,X})^{-1})\right)'}{\partial\boldsymbol{\varrho}}\right){\bf C}_{\boldsymbol{\theta}}.\label{infoexpression2}
\end{split}
\end{gather}
The expectation in (\ref{infoexpression2}) is
\begin{gather}
\begin{split}
&\mathbb{E}\frac{\partial\left({\rm vec}\;(\Psi_{\boldsymbol{\theta}}({\bf X,X})^{-1}\delta({\bf X}) \delta({\bf X})'\Psi_{\boldsymbol{\theta}}({\bf X,X})^{-1})\right)'}{\partial\boldsymbol{\varrho}}\\
&=\frac{\partial({\rm vec}\;(\Psi_{\boldsymbol{\theta}}({\bf X,X})^{-1})'}{\partial\boldsymbol{\varrho}}(\mathbb{E}\delta({\bf X}) \delta({\bf X})'\Psi_{\boldsymbol{\theta}}({\bf X,X})^{-1}\otimes{\bf I}_n)\\
&\quad+\frac{\partial({\rm vec}\;(\Psi_{\boldsymbol{\theta}}({\bf X,X})^{-1})'}{\partial\boldsymbol{\varrho}}({\bf I}_n\otimes\mathbb{E}\delta({\bf X}) \delta({\bf X})'\Psi_{\boldsymbol{\theta}}({\bf X,X})^{-1})\\
&=-2{\bf C}'_{\boldsymbol{\theta}} (\Psi_{\boldsymbol{\theta}}({\bf X,X})^{-1}\otimes \Psi_{\boldsymbol{\theta}}({\bf X,X})^{-1}).\label{infoexpression3}
\end{split}
\end{gather}
Plugging (\ref{infoexpression3}) into (\ref{infoexpression2}) gives
\begin{gather}
\begin{split}
\mathcal{I}_{33}=\mathcal{I}(\boldsymbol{\varrho},\boldsymbol{\varrho})=\frac{1}{2}{\bf C}'_{\boldsymbol{\theta}} (\Psi_{\boldsymbol{\theta}}({\bf X,X})^{-1}\otimes \Psi_{\boldsymbol{\theta}}({\bf X,X})^{-1}){\bf C}_{\boldsymbol{\theta}}.\nonumber
\end{split}
\end{gather}
Using partitioned matrix inverse results \cite{harville} and noting that  ${\bf c}_2$, $\mathcal{I}_{21}$, $\mathcal{I}_{12}$, $\mathcal{I}_{31}$, and $\mathcal{I}_{13}$ are matrices of zeros
gives (\ref{paramErrExpr1}).
\end{proof}
Now, the expressions in Lemma \ref{paramProp} are used to prove Theorem \ref{paramThm}.
The first term on the right-hand side of (\ref{paramErrExpr1}) can be bounded above as
\begin{equation}
\begin{split}
{\bf c}'_1\mathcal{I}^{-1}_{11}{\bf c}_1 &= {\bf c}'_1\left( {\bf H(X)}'\Psi_{\boldsymbol{\theta}}({\bf X,X})^{-1}{\bf H(X)} \right)^{-1}{\bf c}_1\\
& \le \frac{\lambda_{\rm max}( \Psi_{\boldsymbol{\theta}} ({\bf X},{\bf X})) } {\lambda_{\rm min}\left({\bf H}({\bf X})'{\bf H}({\bf X})\right)}  \left\|{\bf c}_1\right\|_2^2\le \frac{n\sup_{{\bf u},{\bf v}\in\Omega}\Psi_{\boldsymbol{\theta}}({\bf u},{\bf v})} {\lambda_{\rm min}\left({\bf H}({\bf X})'{\bf H}({\bf X})\right)}  \left\|{\bf c}_1\right\|_2^2.\nonumber
\end{split}
\end{equation}
The eigenvalue $\lambda_{\rm min}\left({\bf H}({\bf X})'{\bf H}({\bf X})\right)$ has approximation
\begin{gather}
\lambda_{\rm min}\left({\bf H}({\bf X})'{\bf H}({\bf X})\right)=\lambda_{\rm min}\left(\sum_{i=1}^n {\bf h}({\bf x}_i){\bf h}({\bf x}_i)'\right)\approx n\lambda_{\rm min}\left(\int {\bf h}({\bf y}){\bf h}({\bf y})'{\rm d}F({\bf y})\right)=n s_1,\label{boundBound}
\end{gather}
where $F$ denotes the large sample distribution of the input locations ${\bf X}$, $s_1\ge 0$, and $s_1>0$ unless ${\bf h}({\bf y})'{\bf a}=0$ with probability $1$ with respect to the large sample distribution $F$ for some ${\bf a}\ne 0$.
Giving approximate upper bound to the first term on the right-hand side of (\ref{paramErrExpr1})
\begin{gather}
{\bf c}'_1\mathcal{I}^{-1}_{11}{\bf c}_1\le\frac{\sup_{{\bf u},{\bf v}\in\Omega}\Psi_{\boldsymbol{\theta}}({\bf u},{\bf v})} {s_1}  \left\|{\bf c}_1\right\|_2^2,\label{term1Bound}
\end{gather}
where $s_1$ is implicitly defined in (\ref{boundBound}) and the probability of the inequality being violated by more than $\varepsilon>0$ goes to zero as $n\to\infty$.
The second term on the right-hand side of the approximate parameter estimation error expression (\ref{paramErrExpr1}) has
\begin{gather}
\begin{split}
&{\bf c}'_3(\mathcal{I}_{33}-\mathcal{I}_{32}\mathcal{I}_{22}^{-1}\mathcal{I}_{23})^{-1}{\bf c}_3\le\|{\bf c}_3\|_2^2\left/\lambda_{\rm min}\left(\mathcal{I}_{33}-\mathcal{I}_{32}\mathcal{I}_{22}^{-1}\mathcal{I}_{23}\right).\right.\label{c3Bound}
\end{split}
\end{gather}
Note that 
\begin{gather}
\begin{split}
&\mathcal{I}_{33}-\mathcal{I}_{32}\mathcal{I}_{22}^{-1}\mathcal{I}_{23}\\
&=\frac{1}{2\sigma^4}{\bf C}'_{\boldsymbol{\theta}}\left(\left(\Phi_{\boldsymbol{\varrho}}({\bf X},{\bf X})^{-1}\otimes\Phi_{\boldsymbol{\varrho}}({\bf X},{\bf X})^{-1}\right)-\frac{1}{n}\left({\rm vec}\;\Phi_{\boldsymbol{\varrho}}({\bf X},{\bf X})^{-1}\right)\left({\rm vec}\;\Phi_{\boldsymbol{\varrho}}({\bf X},{\bf X})^{-1}\right)'\right){\bf C}_{\boldsymbol{\theta}}.\nonumber
\end{split}
\end{gather}
The matrix inside the quadratic form has eigenvector ${\bf u}_1={\rm vec}\;\Phi_{\boldsymbol{\varrho}}({\bf X},{\bf X})/\|{\rm vec}\;\Phi_{\boldsymbol{\varrho}}({\bf X},{\bf X})\|_2$ with corresponding eigenvalue $0$.
So, the minimum eigenvalue of the above expression can be bounded below by
\begin{gather}
\begin{split}
&\frac{1}{2\sigma^4}\lambda_{\rm min}\left({\bf C}'_{\boldsymbol{\theta}}({\bf I}_{n^2}-{\bf u}_1{\bf u}'_1){\bf C}_{\boldsymbol{\theta}}\right)\\
&\times\lambda_{2}\left(\left(\Phi_{\boldsymbol{\varrho}}({\bf X},{\bf X})^{-1}\otimes\Phi_{\boldsymbol{\varrho}}({\bf X},{\bf X})^{-1}\right)-\frac{1}{n}\left({\rm vec}\;\Phi_{\boldsymbol{\varrho}}({\bf X},{\bf X})^{-1}\right)\left({\rm vec}\;\Phi_{\boldsymbol{\varrho}}({\bf X},{\bf X})^{-1}\right)'\right),\nonumber
\end{split}
\end{gather}
where $\lambda_{2}\left(\cdot\right)$ denotes the \emph{second} smallest eigenvalue of its argument. By Weyl's theorem \cite{ipsen}, the second smallest eigenvalue of the perturbed matrix can be bounded below by 
\begin{gather}
\lambda_{\rm min}\left(\left(\Phi_{\boldsymbol{\varrho}}({\bf X},{\bf X})^{-1}\otimes\Phi_{\boldsymbol{\varrho}}({\bf X},{\bf X})^{-1}\right)\right)=1/\lambda_{\rm max}\left(\Phi_{\boldsymbol{\varrho}}({\bf X},{\bf X})\right)^2\ge 1/(n\sup_{{\bf u},{\bf v}\in\Omega}\Phi_{\boldsymbol{\varrho}}({\bf u},{\bf v}))^2.\label{boundStar}
\end{gather}
Further, 
\begin{gather}
\begin{split}
&{\bf C}'_{\boldsymbol{\theta}}({\bf I}_{n^2}-{\bf u}_1{\bf u}'_1){\bf C}_{\boldsymbol{\theta}}\\
&=\sigma^4\left[\sum_{i,j}\frac{\partial\Phi_{\boldsymbol{\varrho}}({\bf x}_i,{\bf x}_j)}{\partial\boldsymbol{\varrho}}\frac{\partial\Phi_{\boldsymbol{\varrho}}({\bf x}_i,{\bf x}_j)}{\partial\boldsymbol{\varrho}'}\right.\\
&\left.-\frac{1}{\|{\rm vec}\;\Phi_{\boldsymbol{\varrho}}({\bf X},{\bf X})\|_2^2}\left(\sum_{i,j}\frac{\partial\Phi_{\boldsymbol{\varrho}}({\bf x}_i,{\bf x}_j)}{\partial\boldsymbol{\varrho}}\Phi_{\boldsymbol{\varrho}}({\bf x}_i,{\bf x}_j)\right)\left(\sum_{i,j}\frac{\partial\Phi_{\boldsymbol{\varrho}}({\bf x}_i,{\bf x}_j)}{\partial\boldsymbol{\varrho}}\Phi_{\boldsymbol{\varrho}}({\bf x}_i,{\bf x}_j)\right)'\right]\\
&\approx n^2\sigma^4\left[\int\frac{\partial\Phi_{\boldsymbol{\varrho}}({\bf x},{\bf y})}{\partial\boldsymbol{\varrho}}\frac{\partial\Phi_{\boldsymbol{\varrho}}({\bf x},{\bf y})}{\partial\boldsymbol{\varrho}'}{\rm d}F\times F({\bf x},{\bf y})\right.\\
&\left.-\frac{1}{\|\Phi_{\boldsymbol{\varrho}}\|^2_{L_2(F\times F)}}\left(\int \frac{\partial\Phi_{\boldsymbol{\varrho}}({\bf x},{\bf y})}{\partial\boldsymbol{\varrho}}\Phi_{\boldsymbol{\varrho}}({\bf x},{\bf y}) {\rm d}F\times F({\bf x},{\bf y})\right)\left(\int \frac{\partial\Phi_{\boldsymbol{\varrho}}({\bf x},{\bf y})}{\partial\boldsymbol{\varrho}}\Phi_{\boldsymbol{\varrho}}({\bf x},{\bf y}) {\rm d}F\times F({\bf x},{\bf y})\right)'\right]\\
&\succeq n^2\sigma^4 s_2,\label{boundStar2}
\end{split}
\end{gather}
where $F\times F$ denotes the \emph{product measure} \cite{bartle}.
Applying a version of the Cauchy-Schwarz inequality for random vectors in $L_2(F\times F)$, 
provides 
$s_2\ge 0$ and $s_2>0$, unless $\frac{\partial\Phi_{\boldsymbol{\varrho}}({\bf x},{\bf y})}{\partial\boldsymbol{\varrho}'}{\bf a}=\Phi_{\boldsymbol{\varrho}}({\bf x},{\bf y})b$ with probability 1 with respect to the large sample distribution $F\times F$ for some $\left({\bf a}'\;\;b\right)'\ne {\bf 0}$ \cite{tripathi}.
Combining the bounds in equations (\ref{c3Bound}), (\ref{boundStar}), and (\ref{boundStar2}), gives
\begin{gather}
{\bf c}'_3(\mathcal{I}_{33}-\mathcal{I}_{32}\mathcal{I}_{22}^{-1}\mathcal{I}_{23})^{-1}{\bf c}_3\le\frac{2\sup_{{\bf u},{\bf v}\in\Omega}\Phi_{\boldsymbol{\varrho}}({\bf u},{\bf v})^2}{s_2}\|{\bf c}_3\|_2^2,\label{term2Bound}
\end{gather}
with $s_2$ implicly defined in equation (\ref{boundStar2}) and the probability of the inequality being violated by more than $\varepsilon>0$ going to zero as $n\to\infty$.
Combining the approximate bounds (\ref{term1Bound}) and (\ref{term2Bound}) gives the result.

\section*{Acknowledgements}
The authors gratefully acknowledge funding from NSF DMS-1621722.

\end{document}